\documentclass[11pt,dvips,twoside,letterpaper]{article}
\usepackage{pslatex}
\usepackage{fancyhdr}
\usepackage{graphicx}
\usepackage{geometry}
\RequirePackage[latin1]{inputenc}

\def\figurename{Figure} 
\makeatletter
\renewcommand{\fnum@figure}[1]{\figurename~\thefigure.}
\makeatother

\def\tablename{Table} 
\makeatletter
\renewcommand{\fnum@table}[1]{\tablename~\thetable.}
\makeatother

\usepackage{amsmath}
\usepackage{amssymb}
\usepackage{amsfonts}
\usepackage{amsthm,amscd}

\newtheorem{theorem}{Theorem}[section]

\newtheorem{proposition}{Proposition}[section]
\theoremstyle{definition}
\newtheorem{definition}{Definition}[section]

\theoremstyle{remark}
\newtheorem{remark}{Remark}[section]

\numberwithin{equation}{section}

\def\P{\mathbb P}
\def\R{\mathbb R}
\def\E{\mathbb E}

\def\E{\mathbb E}
\def\N{\mathbb N}
\def\cal{\mathcal}

\begin{document}
\title{\bfseries\scshape{Obstacle problem for SPDE with nonlinear Neumann boundary condition via reflected generalized backward doubly SDEs}}
\author{\bfseries\scshape Auguste Aman\thanks{Supported by AUF post doctoral grant 07-08, Réf:PC-420/2460}\;\;\thanks{augusteaman5@yahoo.fr,\ Corresponding author.}\\
U.F.R.M.I, Universit\'{e} de Cocody, \\582 Abidjan 22, C\^{o}te d'Ivoire\\
\\\bfseries\scshape N. Mrhardy\thanks{n.mrhardy@ucam.ac.ma}\\
F.S.S.M, Universit\'{e} Cadi Ayyad, \\2390, Marrakech, Maroc}

\date{}
\maketitle \thispagestyle{empty} \setcounter{page}{1}

\begin{abstract}
This paper is intended to give a representation for
stochastic viscosity solution of semi-linear reflected stochastic
partial differential equations with nonlinear Neumann boundary
condition. We use its connection with reflected generalized backward
doubly stochastic differential equations.
\end{abstract}


\vspace{.08in} \noindent \textbf{Keywords}: Backward doubly SDEs, Stochastic PDEs, Obstacle
problem, stochastic viscosity solutions.

\section{Introduction}
 Backward stochastic
differential equations (BSDEs, for short) were introduced by Pardoux
and Peng \cite{PP1} in 1990, and it was shown in various papers that
stochastic differential equations (SDEs) of this type give a
probabilistic representation for solution (at least in the viscosity
sence) of a large class of system of semi-linear parabolic partial
differential equations (PDEs). Thereafter a new class of BSDEs,
called backward doubly stochastic (BDSDEs), was considered by
Pardoux and Peng \cite{PP3}. It seems suitable
for giving a representation for a system of parabolic
stochastic partial differential equations (SPDEs). We refer to
Pardoux and Peng \cite{PP3} for the link between SPDEs
and BDSDEs when the solutions of SPDEs are regular i.e the coefficients are smooth enough (at least in $C^3$).
The more general situation is much more delicate to treat
because of the difficulties of extending the notion of viscosity solutions to SPDEs.

The notion of viscosity solution for PDEs was introduced by Crandall, Ishii
and Lions \cite{CL} for certain first-order Hamilton-Jacobi
equations. Today this theory becomes an important tool in many
applied fields, especially in optimal control theory and numerous
subjects related to it.

The stochastic viscosity solution for semi-linear SPDEs was
introduced firstly by Lions and Souganidis in \cite{LS}.
They use the so-called "stochastic characteristic" to remove the
stochastic integrals from a SPDEs. A few years later, two others definitions of stochastic viscosity solution of SPDEs are considered by
Buckdahn and Ma respectively in \cite{BM1,BM2} and \cite{BM3}. In \cite{BM1,BM2}, they used the "Doss-Sussman"
 transformation to connect the stochastic viscosity solution of SPDEs with the solution of associated BDSDEs.
 In \cite{BM3}, they introduced the stochastic viscosity solution by using the notion
 of stochastic sub and super jets. Next, in order to give a representation for viscosity solution of  SPDEs with nonlinear Neumann boundary condition,
 Boufoussi et al. \cite{Bal} introduced the so-called generalized BDSDEs. They refer
 the first technique (Doss-Sussman transformation) of Buckdhan and Ma \cite{BM1,BM2}.

Motivated by the work of Boufoussi et al. \cite{Bal} and employing the penalization method, we aim to establish the existence of viscosity solution for semi-linear reflected SPDEs with nonlinear Neumann boundary condition of the form:
\begin{eqnarray*}
\mathcal{OP}^{(f,\phi,g,h,l)}\left\{
\begin{array}{l}
(i)\,\displaystyle \min\left\{u(t,x)-h(t,x),\; -\frac{\partial
u(t,x)}{\partial t}-[
Lu(t,x)+f(t,x,u(t,x),\sigma^{*}(x)D_{x}u(t,x))]\right.\\\\
\left.\,\,\,\,\,\,\,\,\,\,\,\,\,\,\,\,\,\ -
g(t,x,u(t,x))\dot{B}_{s}\right\}=0,\,\,\
(t,x)\in[0,T]\times\Theta\\\\
(ii)\,\displaystyle\frac{\partial u}{\partial
n}(t,x)+\phi(t,x,u(t,x))=0,\,\,\ (t,x)\in[0,T]\times\partial\Theta, \\\\
(iii)\,u(T,x)=l(x),\,\,\,\,\,\,\ x\in\overline{\Theta}
\end{array}\right.
\end{eqnarray*}
where $\dot{B}$ denotes white noise, $L$ is the infinitesimal generator of a diffusion process $X$, $\Theta$ is a connected bounded
domain included in $\R^{d}$, $(d\geq 1)$; and $f,\, g,\, \phi,\, l, h$ are some measurable functions.

More precisely, we give a direct links with the following reflected generalized BDSDE:
\begin{eqnarray*}
Y_{t}&=&\xi+\int_{t}^{T}f(s,Y_{s},Z_{s})ds+\int_{t}^{T}\phi(s,Y_{s})dA_{s}+\int_{t}^{T}g(s,Y_{s})\,\overleftarrow{dB}_{s}\\
&&-\int_{t}^{T}Z_{s}dW_{s}+K_{T}-K_t,\,\ 0\leq t\leq
T,\label{a1}
\end{eqnarray*}
where $\xi$ is the terminal value, $A$ is a positive real-valued
increasing process and  $dW$ and $\overleftarrow{dB}$ denote respectively the classical
forward and backward It\^{o} integral. Our work generalize \cite{Ral} in where authors treat
deterministic reflected PDEs with nonlinear Neumann boundary conditions i.e $g\equiv0$ and the second appears in \cite{Bal} where
the non reflected SPDE with nonlinear Neumann boundary condition is considered.

The present paper is organized as follows: An existence and
uniqueness result to large class of reflected generalized BDSDEs is
shown in Section 2. Section 3 is devoted to give a definition of a
reflected stochastic solution to SPDEs and establishes its existence result.

\section{Reflected generalized backward doubly stochastic differential equations}
\subsection{Notation, assumptions and definition.}

The scalar product of the space $\R^{d} (d\geq 2)$ will be denoted
by $<.,.>$ and the associated Euclidian norm  by $\|.\|$.

In what follows let us fix a positive real number
$T>0$. First of all $\{W_{t}, 0\leq t\leq T\}$ and $\{B_{t},\ 0\leq
t\leq T\}$ are two mutually independent standard Brownian motions
with values respectively in $\mbox{I\hspace{-.15em}R}^{d}$ and
$\mbox{I\hspace{-.15em}R}^{\ell}$, defined respectively on the two
complete probability spaces $(\Omega,\mathcal{F},{\P})$ and
$(\Omega',\mathcal{F}',{\P}')$. Let ${\bf
F}^{B}=\{\mathcal{F}^{B}_{t,T}\}_{t\geq 0}$ denote a retrograde filtration generated by $B$, augmented by the $\P$-null sets of
$\mathcal{F}$; and let $\mathcal{F}^{B}_T=\mathcal{F}^{B}_{0,T}$. We also
consider the following family of $\sigma$-fields:
\begin{eqnarray*}
\mathcal{F}^{W}_{t}=\sigma\{W_{s},0\leq s\leq t\}.
\end{eqnarray*}

Next we consider the product space $(\overline{\Omega},\overline{\mathcal{F}},\overline{\P})$ where
\begin{eqnarray*}
\overline{\Omega}=\Omega\times\Omega',\,\,\overline{\mathcal{F}}=\mathcal{F}\otimes\mathcal{F}'\,\,
\mbox{ and}\,\, \overline{\P}=\P\otimes\P'.
\end{eqnarray*}
For each $t\in[0,T]$, we define
\begin{eqnarray*}
\mathcal{F}_{t}=\mathcal{F}_{t,T}^{B}\otimes\mathcal{F}^{W}_{t}\vee\overline{\mathcal{N}}\; \mbox{and}\;\mathcal{G}_{t}=\mathcal{F}_{T}^{B}\otimes\mathcal{F}^{W}_{t}\vee\overline{\mathcal{N}}.
\end{eqnarray*}
where $\overline{\mathcal{N}}$ denotes all the $\overline{\P}$-null sets in
$\overline{\mathcal{F}}$.

The collection ${\bf F}= \{\mathcal{F}_{t},\ t\in
[0,T]\}$ is neither increasing nor decreasing and it does not
constitute a filtration. However, $(\mathcal{G}_{t})$ is a filtration.

Further, we assume that random variables
$\xi(\omega),\;\omega\in \Omega$ and
$\zeta(\omega'),\; \omega'\in \Omega'$ are considered as
random variables on $\overline{\Omega} $ via the following identification:
\begin{eqnarray*}
\xi(\omega,\omega')=\xi(\omega);\,\,\,\,\,\zeta(\omega,\omega')=\zeta(\omega').
\end{eqnarray*}

In the sequel, let\
$\{A_t,\ 0\leq t\leq T\}$\ be a continuous, increasing and
${\bf F}$-adapted real valued  process such that\ $A_0=0.$

For
any $d\geq 1$, we consider the following spaces of processes:
\begin{enumerate}
\item $M^{2}(0,T,\R^{d})$ denote the Banach space of all
equivalence classes (with respect to the measure $d\mathbb{P}\otimes
dt$) where each equivalence class contains an d-dimensional jointly
measurable stochastic process $\displaystyle{\varphi_{t};
t\in[0,T]}$, such that: for all $\mu>0$
\begin{description}
\item $(i)$ $\displaystyle{\|\varphi\|^{2}_{M^{2}}=\E\int ^{T}_{0}e^{\mu A_t}|\varphi_{t}|^{2}dt<\infty}$;

\item $(ii)$ $\varphi_t$ is ${\mathcal{F}}_{t}$-measurable , for any $t\in [0,T]$.
\end{description}
\item  $S^{2}([0,T],\R)$ is the set of one
dimensional continuous stochastic processes which verify: for all $\mu$
\begin{description}
\item $(iii)$ $\displaystyle{\|\varphi\|^{2}_{S^{2}}=\E\left(\sup_{0\leq t\leq T}e^{\mu A_t}|\varphi_{t}|
^{2}+\int^T_0e^{\mu A_s}|\varphi_s|^{2}dA_s\right)<\infty}$;

\item $(iv)$ $\varphi_t$ is ${\mathcal{F}}_{t}$-measurable , for any $t\in [0,T]$.
\end{description}
\end{enumerate}
In addition, we give the following assumptions on the data $(\xi,f,g,\phi, S)$:
\begin{itemize}
\item[$({\bf H1})$] $\xi$ is an $\mathcal{F}_{T}$-measurable, square integrable random variable; such  that for all $\mu>0$
$$ \mathbb{E}\left(e^{\mu A_T}|\xi|^2\right) < \infty. $$
\item[$({\bf H2})$]
$f:\Omega\times [0,T]\times\R\times\R^{d}\rightarrow \R$,\
$g:\Omega\times [0,T]\times\R\times\R^{d}\rightarrow \R^{\ell},$\
and $\phi:\Omega\times [0,T]\times\R \rightarrow \R,$ are three
functions verifying:
\begin{itemize}
  \item[$(a)$] There exist $\mathcal{F}_t$-measurable processes $\{f_t,\,
\phi_t,\,g_t:\,0\leq t\leq T\}$ with values in $[1,+\infty)$ such that for any $(t,y,z)\in
[0,T]\times\R\times\R^{d}$, and any $\mu>0$, the following hypotheses
are satisfied for $K>0$:
\begin{eqnarray*}
\left\{
\begin{array}{l}
f(t,y,z),\, \phi(t,y),\,\mbox{and}\, g(t,y,z)\, \mbox{are}\, \mathcal{F}_t\mbox{-measurable processes},\\\\
|f(t,y,z)|\leq f_t+K(|y|+\|z\|),\\\\
|\phi(t,y)|\leq \phi_t+K|y|,\\\\
|g(t,y,z)|\leq g_t+K(|y|+\|z\|),\\\\
\displaystyle \E\left(\int^{T}_{0} e^{\mu
A_t}f_t^{2}dt+\int^{T}_{0}e^{\mu A_t}g_t^{2}dt+\int^{T}_{0}e^{\mu
A_t}\phi_t^{2}dA_t\right)<\infty.
\end{array}\right.
\end{eqnarray*}
\item[$(b)$] There exist constants $c>0,\, \beta<0$ and $0<\alpha<1$ such that for any $(y_1,z_1),\,(y_2,z_2)\in\R\times\R^{d}$,
\begin{eqnarray*}
\left\{
\begin{array}{l}
(i)\, |f(t,y_1,z_1)-f(t,y_2,z_2)|^{2}\leq c(|y_1-y_2|^{2}+\|z_1-z_2\|^{2}),\\\\
(ii)\, |g(t,y_1,z_1)-g(t,y_2,z_2)|^{2}\leq c|y_1-y_2|^{2}+\alpha\|z_1-z_2\|^{2},\\\\
(iii)\; \langle y_1-y_2,\phi(t,y_1)-\phi(t,y_2)\rangle\leq
\beta|y_1-y_2|^{2}.
\end{array}\right.
\end{eqnarray*}
\end{itemize}
\item[(${\bf H3}$)] The obstacle $\left\{  S_{t},0\leq t\leq T\right\}$,
is a continuous, $\mathcal{F}_{t}$-measurable,
real-valued process, satisfying for any $\mu>0$ $$\E\left(  \sup_{0\leq t\leq T}e^{\mu A_t}\left|
S^+_{t}\right| ^{2}\right)  <\infty,$$ and $S_{T}\leq\xi\ a.s.$
\end{itemize}
One of our main goals in this paper is the study of reflected generalized BDSDEs,
\begin{eqnarray}
Y_{t}&=&\xi+\int_{t}^{T}f(s,Y_{s},Z_{s})ds+\int_{t}^{T}\phi(s,Y_{s})dA_s+\int_{t}^{T}g(s,Y_{s},Z_{s})\,\overleftarrow{dB}_{s}\nonumber\\
&&-\int_{t}^{T}Z_{s}dW_{s}+K_{T}-K_t,\,\ 0\leq t\leq
T.\label{a1}
\end{eqnarray}
First of all let us give a definition of the solution of this BDSDEs.
\begin{definition}\label{def}
By a solution of the reflected generalized BDSDE $(\xi,f,\phi,g,S)$
we mean a triplet of processes  $(Y,Z,K)$, which satisfies (\ref{a1})
such that the following holds $\P$- a.s
\begin{description}
\item $(i)$ $(Y,Z)\in S^{2}([0,T];\R)\times M^{2}(0,T;\R^{d})$
\item  $(ii)$ the map $s\mapsto Y_s$ is continuous
\item  $(iii)$ $Y_{t}\geq S_{t},\,\,\,\ 0\leq t\leq T$,\,\
\item  $(iv)$ $K$ is an increasing process such that $K_{0}=0$\ %
and\ $\displaystyle \int_{0}^{T}\left( Y_{t}-S_{t}\right)
dK_{t}=0$.
\end{description}
\end{definition}

In the sequel, $C$ denotes a positive constant which may vary from
one line the other.
\subsection{Comparison theorem}
Let us give this needed comparison theorem related to the generalized
BDSDE associated to $(\xi, f,\phi,g)$ in the form
\begin{eqnarray*}
Y_{t}&=&\xi+\int_{t}^{T}f(s,Y_{s},Z_{s})ds+\int_{t}^{T}\phi(s,Y_{s})dA_s+\int_{t}^{T}g(s,Y_{s},Z_{s})\,\overleftarrow{dB}_{s}\nonumber\\
&&-\int_{t}^{T}Z_{s}dW_{s},\,\ 0\leq t\leq
T.
\end{eqnarray*}
It's proof follows with the same computation as in \cite{Yal}, with slight
modification due to the presence of the integral with respect the
increasing process $A$. So we just repeat the main step.
\begin{theorem}(Comparison theorem for generalized BDSDE)\label{Thm:comp}
Let $(Y,Z)$ and $(Y',Z')$ be the unique solution
of the non reflected generalized BDSDE associated to
$(\xi,f,\phi,g)$ and $(\xi',f',\phi',g)$ respectively. If $\xi\leq
\xi',\; f(t,Y'_t,Z'_t)\leq f'(t,Y'_t,Z'_t)$ and $\phi(t,Y'_t)\leq
\phi^{'}(t,Y'_t)$, then $Y_t\leq Y'_t,\; \forall\, t\in [0, T]$.
\end{theorem}
\begin{proof}
Let us set $\Delta Y=Y-Y'$, $\Delta Z=Z-Z'$ and $(\Delta Y)^{+}=(Y-Y')^{+}$ (with $f^{+}=\sup\{f,0\}$).\newline Using It\^o's formula, we get, for all $0\leq t\leq T$
\begin{align}\label{eq:1}
&\mathbb{E}((\Delta Y_t)^+)^2+\mathbb{E}\int_t^T \|\Delta
Z_s\|^2{\bf 1}_{\{Y_s>Y'_s\}}ds
\nonumber\\
\leq &\E((\xi-\xi')^{+})^2+ 2\mathbb{E}\int_t^T
(\Delta Y_s)^+{\bf 1}_{\{Y_s>Y'_s\}}\left\{f(s,Y_s,Z_s)-f'(s,Y'_s,Z'_s)\right\}ds\nonumber\\
&+2\mathbb{E}\int_t^T
(\Delta Y_s)^+{\bf 1}_{\{Y_s>Y'_s\}}\left\{\phi(s,Y_s)-\phi'(s,Y'_s)\right\}dA_s\nonumber\\
 &+\mathbb{E}\int_t^T\left\|g(s,Y_s,Z_s)-g(s,Y'_s,Z'_s)\right\|^2{\bf 1}_{\{Y_s>Y'_s\}}ds,
\end{align}
where ${\bf 1}_{\Gamma}$ denotes the characteristic function of a
given set $\Gamma\in {\bf F}$ defined by

${\bf 1}_{\Gamma}(\omega)=\left\{
\begin{array}{l}
1\ \mbox{if}\ \omega\in \Gamma,\\
0\ \mbox{if}\ \omega\in \Gamma.
\end{array}\right.
$\newline
From $({\bf H2})(b)$, we have
\begin{eqnarray*}
2(\Delta Y_s)^+\left\{f(s,Y_s,Z_s)-f'(s,Y'_s,Z'_s)\right\}
&\leq & 2(\Delta Y_s)^+\left\{f(s,Y_s,Z_s)-f(s,Y'_s,Z'_s)\right\}\\
&\leq &(\frac{1}{\varepsilon}+2c)((\Delta
Y_s)^+)^2+\varepsilon c\|\Delta Z_s\|^2,
\end{eqnarray*}
\begin{eqnarray*}
2(\Delta Y_s)^+\left\{\phi(s,Y_s)-\phi'(s,Y'_s)\right\} &\leq &
2(\Delta Y_s)^+\left\{\phi(s,Y_s)-\phi(s,Y'_s)\right\}\\
&\leq & 2\beta((\Delta Y_s)^+)^2
\end{eqnarray*}

and
\begin{eqnarray*}
\left\|g(s,Y_s,Z_s)-g(s,Y'_s,Z'_s)\right\|^2{\bf 1}_{\{Y_s>Y'_s\}}
&\leq&c((\Delta Y_s)^+)^2{\bf 1}_{\{Y_s>Y'_s\}}+\alpha\|\Delta
Z_s\|^2{\bf 1}_{\{Y_s>Y'_s\}}.
\end{eqnarray*}
Plugging these inequalities in (\ref{eq:1}) and choosing
$\displaystyle \varepsilon=\frac{1-\alpha}{2c}$, we conclude that
$$ \mathbb{E}((\Delta Y_t)^+)^2\leq 0$$
which leads to $\Delta Y^+_t=0$ a.s. and so $Y'_t\geq Y_t$ a.s. for
all $t\leq T$.
\end{proof}
\subsection {Existence and uniqueness result}
Our main goal in this section is to prove the following theorem.
\begin{theorem}\label{Th:exis-uniq}
Under the hypotheses $(\bf H1)$-$({\bf H3})$, the reflected
generalized BDSDE \eqref{a1} has a unique solution $(Y,Z,K)$.
\end{theorem}
Before we start proving this theorem, let us establish the same result in case
$g$ do not depends on $Y$ and $Z$.
More precisely, given $g$ such that $$\E\left(\int_0^Te^{\mu A_s}\|g(s)\|^2ds\right)<\infty$$ and $f,\; \phi$ and $\xi$ as above, consider the reflected generalized BDSDE
\begin{eqnarray}
Y_{t}&=&\xi+\int_{t}^{T}f(s,Y_s,Z_s)ds+
\int_{t}^{T}\phi(s,Y_s)dA_s+\int_{t}^{T}g(s)\,\overleftarrow{dB}_{s}-\int_{t}^{T}Z_{s}dW_{s}+K_T-K_t. \nonumber\\
&&\label{h4}
\end{eqnarray}
\begin{proposition}\label{propo}
There exists a unique triplet $(Y,Z,K)$ verifies conditions $(i)$-$(iv)$ of definition \ref{def} and satisfies \eqref{h4}.
\end{proposition}
\begin{proof}
\noindent{\bf Existence}\newline
It is based on a slight adaptation of the penalization method taking account the presence of the backward Itô integral with respect the
Brownian motion $B$.
For each $n\in\N^{*}$, we set
\begin{eqnarray}\label{a0}
f_{n}(s,y,z)=f(s,y,z)+n(y-S_{s})^{-}
\end{eqnarray}
and consider the generalized BDSDE
\begin{eqnarray}
Y_{t}^{n}&=&\xi+\int_{t}^{T}f_n(s,Y_s^n,Z_s^n)ds+
\int_{t}^{T}\phi(s,Y_{s}^{n})dA_s\nonumber\\
&&+\int_{t}^{T}g(s)\,\overleftarrow{dB}_{s}-\int_{t}^{T}Z_{s}^{n}dW_{s}.
\label{h3}
\end{eqnarray}
It is well known (see Boufoussi et al., \cite{Bal}) that generalized BDSDE \eqref{h3} has a
unique solution $(Y^{n},Z^{n})\in S^{2}([0,T];\R)\times M^{2}(0,T;\R^{d})$ such that for each
$n\in\N^{*}$,
\begin{eqnarray*}
\E\left(\sup_{0\leq t\leq
T}e^{\mu A_t}|Y^{n}_{t}|^{2}+\int_0^Te^{\mu A_s}\left\|Z_s^{n}\right\|^2 ds
\right)<\infty.
\end{eqnarray*}
On the other hand, for all $n\geq 0$ and $s,y,z)\in[0, T]\times\R\times\R^d$,
$$f_n(s,y,z)\leq f_{n+1}(s,y,z),$$
which provide by Theorem \ref{Thm:comp}, $Y^n_t\leq Y^{n+1}_t, \, t\in[0,T]$ a.s.
Therefore,  setting  $Y_{t}=\sup_{n}Y_t^n$ we have $$Y_{t}^{n}\nearrow Y_{t}\;\; \mbox{a.s.}$$
{\it Step 1: A priori estimate}\newline
For any $\mu>0$, there exists
$C>0$ such that,
\begin{eqnarray*}
\sup_{n\in\N^{*}}\E\left( \sup_{0\leq t\leq T}e^{\mu A_t}\left|
Y_{t}^{n}\right| ^{2}+\int_{0}^{T}e^{\mu A_s}\left| Y_{s}^{n}\right| ^{2}
dA_s+\int_{0}^{T}e^{\mu A_s}\left\| Z_{s}^{n}\right\| ^{2}
ds+|K^{n}_{T}|^{2}\right)<C
\end{eqnarray*}
where
\begin{eqnarray} K_{t
}^{n}=n\int_{0}^{t}(Y_{s}^{n}-S_{s})^{-}ds , \,\ 0\leq t\leq
T.\label{K}
\end{eqnarray}
Indeed, from It\^{o}'s formula it follows that
\begin{align}\label{b2}
& e^{\mu A_t}\left| Y_{t}^{n}\right|^{2}+\int_{t}^{T}e^{\mu A_s}\left\| Z_{s}^{n}\right\|
^{2}ds \nonumber\\
& \leq e^{\mu A_T}\left| \xi\right| ^{2}+2\int_{t}^{T}e^{\mu
A_s}Y_{s}^{n}f(s,Y_{s}^{n},Z_{s}^{n}) ds
+2\int_{t}^{T}e^{\mu A_s}Y_{s}^{n}\phi(s,Y_{s}^{n})dA_s-\mu\int_{t}^{T}e^{\mu A_s}|Y_{s}^{n}|^2dA_s\nonumber\\
& +\int_{t}^{T}e^{\mu
A_s}\|g(s)\|^{2}ds+2\int_{t}^{T} e^{\mu
A_s}S_{s}dK_{s}^{n}+ 2\int_{t}^{T}e^{\mu A_s}\langle Y_{s}^{n}
,g(s,Y_{s}^{n},Z_{s}^{n})\overleftarrow{dB}_{s}\rangle\nonumber\\
& -2\int_{t}^{T}e^{\mu A_s}\langle
Y_{s}^{n},Z_{s}^{n}dW_{s}\rangle,
\end{align}
where we have used
$\displaystyle{\int_{t}^{T}e^{\mu A_s}(Y^{n}_{s}-S_{s})dK^{n}_{s}\leq 0}$ and
the fact that
\begin{eqnarray*}
\int_{t}^{T}e^{\mu A_s}Y^{n}_{s}dK^{n}_{s}&=&
\int_{t}^{T}e^{\mu A_s}(Y^{n}_{s}-S_{s})dK^{n}_{s}+\int_{t}^{T}e^{\mu A_s}S_{s}dK^{n}_{s}\leq
\int_{t}^{T}e^{\mu A_s}S_{s}dK^{n}_{s}.
\end{eqnarray*}
Using $({\bf H2})$ and the elementary inequality $2ab\leq \gamma
a^{2}+\frac{1}{\gamma} b^2,\ \forall\gamma>0$,
\begin{eqnarray*}
2Y_{s}^{n}f(s,Y_{s}^{n},Z_{s}^{n})&\leq&(c\gamma_1+\frac{1}{\gamma_1})|Y^{n}_{s}|^{2}
+2c\gamma_1\|Z^{n}_{s}\|^{2}+2\gamma_1 f_s^{2},\\
2Y_{s}^{n}\phi(s,Y_{s}^{n})&\leq& (\gamma_2-2|\beta|)|Y^{n}_{s}|^{2}+\frac{1}{\gamma_2}\phi_s^{2}.
\end{eqnarray*}
Taking expectation in both sides of the inequality (\ref{b2}) and
choosing $\displaystyle{\gamma_1=\frac{1-\alpha}{6c}}$ and
$\displaystyle\gamma_2-\mu=|\beta|$ we obtain for all
$\varepsilon >0$
\begin{align}\label{a2}
&
\E (e^{\mu A_t}\left|Y_{t}^{n}\right|^{2})+|\beta|\E\int_{t}^{T}e^{\mu A_s}\left|Y_{s}^{n}\right|^{2}dA_s
+\frac{1-\alpha}{6}\E\int_{t}^{T}e^{\mu A_s}\left\|Z_{s}^{n}\right\|^{2}ds \nonumber\\
&\leq
C\E\left\{e^{\mu A_T}|\xi|^{2}+\int_t^Te^{\mu A_s}|Y^{n}_{s}|^{2}ds+\int_t^Te^{\mu A_s}f_s^{2}ds+
\int_t^Te^{\mu A_s}\phi_s^{2}dA_s+\int_t^Te^{\mu A_s}\|g(s)\|^{2}ds\right\}\nonumber\\
& +\frac{1}{\varepsilon}\E\left(\sup_{0\leq s\leq
T}(e^{\mu A_s}S_{s}^+)^2\right)+\varepsilon\E \left(K_T^n-K_{t}^{n}\right)^2.
\end{align}
On the other hand, we get from (\ref{h3}) that for all $0 \leq t
\leq T$,
\begin{equation}\label{K:n}
K_T^n-K_t^n=Y_t^n-\xi-\int_t^T f(s,Y_s^n,Z_s^n) ds -\int_t^T
\phi(s,Y_s^n) dA_s-\int_t^T g(s) \overleftarrow{dB}_s+\int_t^T Z_s^n
\ d W_s.
\end{equation}
Then we have
\begin{align}\label{estK}
\E(K^n_T-K^n_t)^2\leq 6\E\left\{e^{\mu A_T}|\xi|^{2}+e^{\mu A_t}|Y^{n}_t|^2+\left|\int_t^Tf(s,Y^{n}_s,Z^n_s)ds\right|^{2}\right.\nonumber\\
\left.+\left|\int_t^T\phi(s,Y^n_s)dA_s\right|^{2}+\left|\int_t^Tg(s)\overleftarrow{dB}_s\right|^{2}+\left|\int_t^TZ^{n}_sdW_s\right|^{2}\right\}.
\end{align}
It follows by H\"{o}lder inequality and the isometry equality, together with assumptions $({\bf H2})(a)$ that
\begin{eqnarray*}
\left|\int_t^Tf(s,Y^{n}_s,Z^n_s)ds\right|^{2}\leq
3(T-t)\int_t^Te^{\mu
A_s}(f_s^2+K^2|Y^{n}_s|^{2}+K^2\|Z^{n}_s\|^{2})ds,
\end{eqnarray*}
and
\begin{eqnarray*}
\E\left|\int_t^TZ^{n}_sdW_s\right|^{2}\leq \E\int_t^Te^{\mu A_s}|Z^{n}_s|^{2}ds.
\end{eqnarray*}
Next, to estimate  $\left|\int_t^T\phi(s,Y^n_s)dA_s\right|^{2}$,
let us assume first that $A_T$ is a bounded  real variable. For any $\mu>0$ given in assumptions $({\bf
H1})$ or $({\bf H2})(a)$, we have
\begin{eqnarray*}
\left|\int_t^T\phi(s,Y^{n}_s)dA_s\right|^{2}&\leq& \left(\int_t^Te^{-\mu A_s}dA_s\right)\left(\int_t^Te^{\mu A_s}|\phi(s,Y^{n}_s)|^{2}dAs\right)\nonumber\\
&\leq& \frac{2}{\mu}\int_t^Te^{\mu A_s}(\phi_s^2+K^2|Y^{n}_s|^{2})dA_s,\label{psi}
\end{eqnarray*}
since
\begin{eqnarray*}
\left(\int_t^Te^{-\mu A_s}dA_s\right)\leq \frac{1}{\mu}[1-e^{-\mu A_T}]\leq \frac{1}{\mu}.
\end{eqnarray*}
The general case then follows from Fatou's lemma.

Therefore, from $(\ref{estK})$ together with the previous
inequalities, there exists a constant independent of $A_T$ such that
\begin{align}\label{kn1}
\E(K_T^n-K^n_t)^2&\leq C\E\left\{e^{\mu A_T}|\xi|^{2}+e^{\mu
A_t}|Y^{n}_t|^2+\int_t^Te^{\mu A_s}f_s^{2}ds+ \int_t^Te^{\mu
A_s}\phi_s^{2}dA_s+\int_t^Te^{\mu A_s}\|g(s)\|^{2}ds\right.\nonumber\\
&\left.+\int_t^T e^{\mu
A_s}\left|Y_s^n\right|^2ds+\E\left(\sup_{0\leq s\leq t}e^{\mu
A_s}(S_{s}^+)^2\right)+\int_t^T e^{\mu A_s}
\left|Y_s^n\right|^2dA_s+\int_t^T e^{\mu A_s}\|Z_s^n\|^2ds\right\}.
\end{align}
Recalling again $(\ref{a2})$ and taking $\varepsilon$ small enough
such that $\varepsilon C<\min\{1,|\beta|,\frac{1-\alpha}{6}\}$, we
obtain
\begin{align*}
&
\E\left( e^{\mu A_t}\left|Y_{t}^{n}\right|^{2}+\int_{t}^{T}e^{\mu A_s}\left|Y_{s}^{n}\right|^{2}dA_s
+\int_{t}^{T}e^{\mu A_s}\left\|Z_{s}^{n}\right\|^{2}ds\right) \\
&\leq
C\E\left\{e^{\mu A_T}|\xi|^{2}+\int_t^Te^{\mu A_s}|Y^{n}_{s}|^{2}ds+\int_t^Te^{\mu A_s}f_s^{2}ds+
\int_t^Te^{\mu A_s}\phi_s^{2}dA_s\right.\\
&\left.+\int_t^Te^{\mu A_s}\|g(s)\|^{2}ds
+\left(\sup_{0\leq s\leq T}e^{\mu A_s}(S_{s}^+)^2\right)\right\}
\end{align*}
Consequently, it follows from Gronwall's lemma and (\ref{kn1}) that
\begin{align*}
&
\E\left\{e^{\mu A_t}|Y_{t}^{n}|^{2}+\int_{t}^{T}e^{\mu A_s}\left|Y_{s}^{n}\right|^{2}dA_s
+\int_{t}^{T}e^{\mu A_s}\|Z_{s}^{n}\|^{2}ds+|K_{T}^{n}-K^n_t|^{2}\right\}\\
&\leq
C\E\left\{e^{\mu A_T}|\xi|^{2}+\int_t^Te^{\mu A_s}f_s^{2}ds+\int_t^Te^{\mu A_s}\phi_s^{2}dA_s
+\int_t^Te^{\mu A_s}\|g(s)\|^{2}ds+\sup_{0\leq t\leq T}e^{\mu A_t}(S_{t}^{+})^{2}\right\}.
\end{align*}
Finally, by  application of Burkholder-Davis-Gundy inequality we obtain from $(\ref{b2})$
\begin{eqnarray*}
\E\left\{\sup_{0\leq t\leq T}
e^{\mu A_t}|Y_{t}^{n}|^{2}+\int_{t}^{T}e^{\mu A_s}\|Z_{s}^{n}\|^{2}ds+|K_{T}^{n}|^{2}\right\}
&\leq & C\E\left\{e^{\mu A_T}|\xi|^{2}+\int_t^Te^{\mu A_s}f_s^{2}ds+\int_t^Te^{\mu A_s}\phi_s^{2}dA_s\right.\nonumber\\
&&+\left.\int_t^Te^{\mu A_s}\|g(s)\|^{2}ds+\sup_{0\leq t\leq
T}e^{\mu A_t}(S_{t}^{+})^{2}\right\},
\end{eqnarray*}
which end the step.

{\it Step 2}: For each $n\in\N^{*}$,
\begin{eqnarray*}
\E\left(\sup_{0\leq t\leq T }\left|\left( Y_{t}^{n}-S_{t}\right)
^{-}\right| ^{2}\right)\longrightarrow 0,\,\,\ \mbox{as}\,\ n
\longrightarrow \infty.
\end{eqnarray*}
Indeed, since $Y^{n}_{t}\geq Y^{0}_{t}$, we can w.l.o.g. replace $S_{t}$ by
$S_{t}\vee Y^{0}_{t}$, i.e. we may assume that\\ $\E(\sup_{0\leq
t\leq T}e^{\mu A_t}S_{t}^{2})<\infty$. We want to compare a.s. $Y_t$ and $S_t$
for all $t\in[0,T]$, while we do not know yet if $Y$ is a.s.
continuous.

In fact, denoting
$$
\left\{
\begin{array}{ll}
&\displaystyle \overline{\xi}:=\xi+\int_{0}^{T}g\left(
s\right)  \overleftarrow{dB}_{s}\\
& \displaystyle\overline{S}_{t}:=S_{t}+\int_{0}^{t}g\left(
s\right)  \overleftarrow{dB}_{s}\\
& \displaystyle\overline{Y}_{t}^{n}:=Y_{t}^{n}+\int_{0} ^{t}g\left(
s\right)  \overleftarrow{dB}_{s}
\end{array}
\right.
$$
we have
\begin{equation}\label{Ynbar}
\overline{Y}_{t}^{n}=\overline{\xi}+\int_{t}^{T}f\left(
s,Y_s^n,Z_s^n\right) ds+n\int_{t}^{T}\left(\overline{Y}_{s}^{n}
-\overline{S_{s} }\right) ^{-}ds+\int_{t}^{T}\phi\left(s,
Y_s^n\right) dA_s-\int_{t}^{T}Z_{s}^{n} dW_{s}.
\end{equation}
Set $\sup_{n}\overline{Y}_{t}^{n}=\overline{Y}_{t}$; then $Y_t=\overline{Y}_{t}-\int^{t}_{0}g(s)dB_s$.

Let
\begin{eqnarray*}
\widetilde{Y}_{t}^{n} &=&\overline{S}_{T}+\int_{t}^{T}f\left(
s,Y_s^n,Z_s^n\right)
ds+n\int_{t}^{T}(\overline{S}_{s}-\widetilde{Y}_{s}^{n})ds
+\int_{t}^{T}\phi\left(s, Y_s^n\right)
dA_s-\int_{t}^{T}\widetilde{Z} _{s}^{n}dW_{s}.
\end{eqnarray*}
Since $\overline{S}_T\leq \overline{\xi}$, then Theorem \ref{Thm:comp} show that, for $t\in[0,T],\, \widetilde{Y}_{t}^{n}\leq\overline{Y}^{n}_t$ a.s.

Let $\sigma$ be a $\mathcal{G}_{t}$-stopping time, and put $\nu=\sigma\wedge T$. The sequence $(\widetilde{Y}_n)$ satisfies then the equality
\begin{eqnarray}
\widetilde{Y}_{\nu }^{n} &=&\E\left\{ e^{-n(T-\nu)}\overline{S}_{T}
+\int^{ T}_{\nu }e^{-n(s-\nu)}f(s,Y_{s}^{n},Z_{s}^{n})ds+n\int^{ T}_{\nu }e^{-n(s-\nu)}\overline{S}_{s}ds \right.\nonumber \\
&&\left.+\int^{ T}_{\nu }e^{-n(s-\nu)}\phi(s,Y_{s}^{n})dA_s\mid
{\cal G} _{\nu }\right\}. \label{c'2}
\end{eqnarray}
First, according to previous work (see \cite{Kal}), $\displaystyle{e^{-n(T-\nu)}\overline{S}_{T}+n\int^{T}_{\nu }e^{-n(s-\nu)}\overline{S}_{s}ds}$ converge to $\overline{S}_{\nu}$ a.s. and in $L^2(\overline{\Omega})$ and it conditional expectation converges also
in $L^2(\overline{\Omega})$.

Moreover,
\begin{eqnarray*}
\E\left(\int^{ T}_{\nu }e^{-n(s-\nu)}f(s,Y_{s}^{n},Z_{s}^{n})ds\right)^{2}&\leq&
\frac{1}{2n}\E\left( \int_{0 }^{T}|
f(s,Y_{s}^{n},Z_{s}^{n})|^{2} ds\right)\\
&\leq& \frac{C}{2n}\E\left( \int_{0}^{T}
e^{\mu A_s}(f_s^{2}+|Y_{s}^{n}|^{2}+\|Z_{s}^{n}\|^{2}) ds\right),
\end{eqnarray*}
and
\begin{eqnarray*}
\E\left(\int^{T}_{\nu}e^{-n((s-\nu)}\phi(s,Y_{s}^{n})dA_s\right)^{2}
&\leq& \E\left[\left(\int^{T}_{\nu}e^{-[2n(s-\nu)+\mu A_{s}]}dA_s\right)\left(
\int^{T}_{\nu}e^{\mu A_s}| \phi(s,Y_{s}^{n})|^{2} dA_s\right)\right]\\
&\leq &\frac{1}{\mu}(1-e^{-\mu A_T})\E\left(
\int_{0}^{T}e^{\mu A_s}(\phi_s^2+K|Y_{s}^{n}|^{2}) dA_s\right).
\end{eqnarray*}
Hence applying Lebesgue dominated Theorem
\begin{eqnarray*}
\E\left[\int^{ T}_{\nu }e^{-n(s-\nu)}f(s,Y_{s}^{n},Z_{s}^{n})ds+\int^{T}_{\nu}e^{-n(\nu-s)}\phi(s,Y_{s}^{n})dA_s|\mathcal{G}_{\nu}\right]\rightarrow 0
\end{eqnarray*}
in $L^{2-\delta}(\overline{\Omega})$, for $\delta >0$ arbitrary taken as $n\rightarrow 0$, which implies that this convergence follows in $L^{1}(\overline{\Omega})$.

Consequently,
\begin{eqnarray*}
\widetilde{Y}_{\nu }^{n}\rightarrow \overline{S}_{\nu}\,\;\;\;\;\;\mbox{in}\,\;\;\;\;\;L^{1}(\overline{\Omega})\, \;\;\;\;\mbox{as}\,\;\;\; \;\;\;\;\; n\rightarrow\infty.
\end{eqnarray*}
Therefore,
${Y}_{\nu}\geq {S}_{\nu}$ a.s. From this and the section theorem (see \cite{DM}, p.220) we deduce that $Y_t\geq S_t$ for every $t\in[0,T ]$ and then $(Y_{t}^{n}-S_{t})^{-}$ converge to zero, a.s., for $t\in[0,T]$. Since $\displaystyle{(Y_{t}^{n}-S_{t})^{-}\leq (S_{t}-Y_{t}^{0})^{+}\leq
\left| S_{t}\right| +\left| Y_{t}^{0}\right|}$, the dominated
convergence theorem ensures that
\begin{eqnarray*}
\lim_{n\longrightarrow +\infty }\E(\sup_{0\leq t\leq T }|\left(
Y_{t}^{n}-S_{t}\right) ^{-}| ^{2})=0.
\end{eqnarray*}
{\it Step 3: Convergence result}\newline
Recalling that $Y_{t}^{n}\nearrow Y_{t}$ $a.s$,
Fatou's lemma and Step 1 ensure
\begin{eqnarray*}
\E\left( \sup_{0\leq t\leq T }e^{\mu A_t}\left| Y_{t}\right| ^{2}\right)
<+\infty.
\end{eqnarray*}
It then follows by dominated convergence that
\begin{eqnarray}
\E\left( \int_{0}^{T }\left| Y_{s}^{n}-Y_{s}\right| ^{2} ds
\right)\longrightarrow 0,\,\,\ \mbox{as}\,\  n\rightarrow
\infty.\label{b9}
\end{eqnarray}

Next, the sequence of processes\ $Z^n$\ converges
in $M^{2}(0,T;\R^{d})$. Indeed, for\ $n \geq p \geq 1$,
It\^{o}'s formula provide
\begin{eqnarray*}
&&\left| Y_{t}^{n}-Y_{t}^{p}\right| ^{2}+\int_{t }^{T }\|
Z_{s}^{n}-Z_{s}^{p}\| ^{2}ds  \nonumber \\
&=&2\int_{t }^{T
}(Y_{s}^{n}-Y_{s}^{p})[f(s,Y_{s}^{n},Z_{s}^{n})-f(s,Y_{s}^{p},Z_{s}^{p})]ds+2\int_{t }^{T
}
(Y_{s}^{n}-Y_{s}^{p})[\phi(s,Y_{s}^{n})-\phi(s,Y_{s}^{p})]dA_s\\
&&-2\int_{t}^{T}\langle Y_{s}^{n}-Y_{s}^{p}
,[Z_{s}^{n}-Z_{s}^{p}]dW_{s}\rangle+2\int_{t}^{T}(Y_{s}^{n}-Y_{s}^{p})(dK_{s}^{n}-dK_{s}^{p})
. \label{b10}
\end{eqnarray*}
From the same step as before, by using again assumptions
$({\bf H2})$, there exists a constant $C>0$, such that
\begin{eqnarray*}
&&\E\left\{\left| Y_{t}^{n}-Y_{t}^{p}\right|^{2}+
\int_{t}^{T}\left| Y_{s}^{n}-Y_{s}^{p}\right| ^{2}dA_s +
\int_{t}^{T}\left\|
Z_{s}^{n}-Z_{s}^{p}\right\| ^{2}ds\right\} \\
&\leq &
C\E\left\{\int_{t}^{T}|Y_{s}^{n}-Y_{s}^{p}|^{2}ds+\sup_{0\leq s\leq
T}\left(Y_{s}^{n}-S_{s}\right)^{-}K_{T}^{p} +\sup_{0\leq s\leq
T}\left( Y_{s}^{p}-S_{s}\right) ^{-} K_{T}^{n}\right\},
\end{eqnarray*}
which, by Gronwall lemma, H\"{o}lder inequality and Step 1 implies
\begin{eqnarray*}
&&\E\left\{\left| Y_{t}^{n}-Y_{t}^{p}\right|^{2} +\int_{t}^{T}\left| Y_{s}^{n}-Y_{s}^{p}\right| ^{2}dA_s+ \int_{t}^{T}\left\|
Z_{s}^{n}-Z_{s}^{p}\right\| ^{2}ds\right\}\nonumber\\ &\leq &
C\left\{\E\left(\sup_{0\leq s\leq T}|\left( Y_{s}^{n}-S_{s}\right)
^{-}|^{2}\right)\right\}^{1/2}+C\left\{ \E\left(\sup_{0\leq s\leq T}|\left(
Y_{s}^{p}-S_{s}\right) ^{-}|^{2}\right)\right\}^{1/2}. \label{b11'}
\end{eqnarray*}
Finally, from Burkh\"{o}lder-Davis-Gundy's inequality, we obtain
\begin{eqnarray*}
\E\left( \sup_{0\leq s\leq T}\left| Y_{s}^{n}-Y_{s}^{p}\right|
^{2}+\int_{t}^{T}\left| Y_{s}^{n}-Y_{s}^{p}\right| ^{2}dA_s+\int_{t}^{T}\left\| Z_{s}^{n}-Z_{s}^{p}\right\| ^{2}ds\right)
\longrightarrow 0,\,\, \mbox{ as }\,  n,p\longrightarrow
\infty,\label{b12}
\end{eqnarray*}
which provides that the sequence of processes $(Y^{n},Z^{n})$ is
Cauchy in the Banach space $S^{2}([0,T];\R)\times
M^{2}(0,T;\R^{d})$. Consequently, there exists a couple $(Y,Z)\in
S^{2}([0,T];\R)\times M^{2}(0,T;\R^{d})$ such that
\begin{eqnarray*}
\E\left\{\sup_{0\leq s\leq T}\left| Y_{s}^{n}-Y_{s}^{{}}\right|
^{2}+\int_{t}^{T}\left| Y_{s}^{n}-Y_{s}\right|
^{2}dA_s+\int_{t}^{T}\left\|Z^{n}_{s}-Z_{s}\right\|^{2}ds\right)
\rightarrow 0,\mbox{ as }n\rightarrow \infty .
\end{eqnarray*}
On the other hand, we rewrite (\ref{K:n}) as
\begin{equation}\label{K:n1}
K_t^n=Y_0^n-Y^n_t-\int_0^t f(s,Y_s^n,Z_s^n) ds -\int_0^t \phi(s,Y_s^n)
dA_s-\int_0^t g(s) \overleftarrow{dB}_s+\int_0^t Z^n_s \ d W_s.
\end{equation}
By the convergence of $Y^n,\ Z^n$ (for a subsequence), the fact that
$f,\phi$ are continuous and
\begin{itemize}
\item $\sup_{n\geq0}|f(s,Y^n_s,Z_s)|\leq f_s+K\left\{(\sup_{n\geq0}|Y_s^n|)+\|Z_s\|\right\}$,
\item $\sup_{n\geq0}|\phi(s,Y^n_s)|\leq \phi_s+K\left\{(\sup_{n\geq0}|Y^n_s|)\right\}$,
\item $\mathbb{E}\int_t^T|f(s,Y^n_s,Z^n_s)-f(s,Y^n_s,Z_s)|^2ds\leq C\mathbb{E}\int_t^T\|Z^n_s-Z_s\|^2ds$
\end{itemize}
we get the existence of a  process $K$ which verifies for all
$t\in[0,T]$
$$\E\left| K_{t}^{n}-K_{t}^{{}}\right| ^{2}\longrightarrow 0$$
and passing to the limit in \eqref{h3}, we have
$$Y_t=\xi+\int_t^Tf(s,Y_s,Z_s)ds+\int_t^T\phi(s,Y_s)dA_s+K_T-K_t+\int_t^T g(s) \overleftarrow{dB}_s-\int_t^TZ_s dW_s, 0\leq t\leq T,$$
$\mathbb{P}$-a.s.

It remains to show that $(Y,Z,K)$ solves the reflected BSDE $(\xi,
f,\phi,g,S)$. In fact, since $(Y^{n},K^{n})$ converges to $(Y,K)$ in
probability uniformly in $t$, $dK^n$ converges to $dK$ in
probability. This implies $\int_{0}^{T}(Y_{s}^{n}-S_{s})dK_{s}^{n}$ converges to $\int_{0}^{T}(Y_{s}-S_{s})dK_{s}$ in probability. We have $Y_{t}\geq
S_{t}$ a.s. for $\in[0,T]$ so that $\int_{0}^{T}(Y_{s}-S_{s})dK_{s}\geq 0$. On the other hand, $\int_{0}^{T}(Y_{s}^{n}-S_{s})dK_{s}^{n}=-n\int_{0}^{T}|
(Y_{s}^{n}-S_{s})^{-}| ^{2}ds \leq 0$. Hence $\int_{0}^{T}(Y_{s}-S_{s})dK_{s}=0$.

\noindent{\bf Uniqueness}\newline
Let us define
\begin{eqnarray*}
\left\{ \left( \Delta Y_{t},\Delta Z_{t},\Delta K_{t}\right) ,\mbox{
}0\leq t\leq T \right\} =\left\{ (Y_{t}-Y_{t}^{\prime
},Z_{t}-Z_{t}^{\prime },K_{t}-K_{t}^{\prime }),\mbox{ }0\leq t\leq T
\right\}
\end{eqnarray*}
where
 $\displaystyle{\left\{
\left( Y_{t},Z_{t},K_{t}\right) ,\mbox{ }0\leq t\leq T
\right\}}$ and $\displaystyle{\left\{ (Y_{t}^{\prime },Z_{t}^{\prime },K_{t}^{\prime }),%
\mbox{ }0\leq t\leq T\right\}}$ denote two solutions of the
reflected generalized BDSDE \eqref{h4}.

It follows again by It\^{o}'s formula that for every $0\leq t\leq T$
\begin{eqnarray*}
\left| \Delta Y_{t}\right| ^{2}+\int_{t}^{T}\|\Delta Z_{s}\| ^{2}ds &=& 2\int_{t}^{T}\Delta
Y_{s}(f(s,Y_{s},Z_{s}^{{}})-f(s,Y_{s}^{\prime },Z_{s}^{\prime
}))ds\nonumber\\
&&+2\int_{t}^{T}\Delta Y_{s}(\phi(s,Y_{s})-\phi(s,Y_{s}^{\prime
}))dA_s\nonumber\\
&&-2\int_{t}^{T}\langle \Delta Y_{s},\Delta Z_{s}dW_{s}\rangle
+2\int_{t}^{T}\Delta Y_{s}d(\Delta K_s).
\end{eqnarray*}
Since
\begin{eqnarray*}
\int_t^T\Delta Y_{s} d(\Delta K_s) \leq 0,
\end{eqnarray*}
and by using similar computation as in the proof of existence, we have
\begin{eqnarray*}
\E\left\{\left| \Delta Y_{t}\right| ^{2}+\int_{t}^{T}|\Delta
Y_{s}|dA_s+\int_{t}^{T}\| \Delta Z_{s}\| ^{2}ds \right\} &\leq
&C\E\int_{t}^{T}|\Delta Y_{s}|^2ds,
\end{eqnarray*}
from which, we deduce that $\displaystyle{\Delta Y_{t}=0}$ and
further $\displaystyle{\Delta Z_{t}=0}$. Moreover since
\begin{eqnarray*}
\Delta K_{t} &=&\Delta
Y_{0}-\Delta Y_{t}-\int_{0}^{t}\left(f(s,Y_{s},Z_{s})-f(s,Y_{s}^{\prime},Z_{s}^{\prime
})\right)ds
-\int_{0}^{t}\left(\phi(s,Y_{s})-\phi(s,Y_{s}^{\prime})\right)dA_s\\
&&+\int_{0}^{t}\Delta Z_{s}dW_{s},
\end{eqnarray*}
we have $\displaystyle{\Delta K_{t}=0}$. The proof is now complete.
\end{proof}

\begin{proof}[Proof of Theorem \ref{Th:exis-uniq}]
For any given $(\bar{Y},\bar{Z})\in S^{2}([0,T];\R)\times M^{2}(0,T;\R^{d})$, let consider the reflected generalized BDSDE:
\begin{equation}
Y_{t}=\xi+\int_{0}^t f(s,{Y}_s,{Z}_s)ds
+\int_{t}^{T}\phi(s,Y_s)dA_s+\int_{t}^{T}g(s,\bar{Y}_s,\bar{Z}_s)\,\overleftarrow{dB}_{s}
-\int_{t}^{T}Z_{s}dW_{s}+K_T-K_t.\label{h'4}
\end{equation}
It follows from Proposition \ref{propo} that reflected generalized BDSDE \eqref{h'4} has a unique solution $(Y,Z,K)$. Therefore, the mapping
$$
\begin{array}{lrlll}
\Phi:&S^{2}([0,T];\R)\times M^{2}(0,T;\R^{d})&\longrightarrow&S^{2}([0,T];\R)\times M^{2}(0,T;\R^{d})\\
&(\bar{Y},\bar{Z})&\longmapsto&(Y,Z)=\Phi(\bar{Y},\bar{Z}).
\end{array}
$$
is well defined.

Next, let $(Y,Z),\ (Y',Z'),\ (\bar{Y},\bar{Z})$ and
$(\bar{Y'},\bar{Z'})\in S^{2}([0,T];\R)\times M^{2}(0,T;\R^{d})$ such that
$(Y,Z)=\Phi(\bar{Y},\bar{Z})$ and $(Y',Z')=\Phi(\bar{Y'},\bar{Z'})$ and set $\Delta \eta=\eta-\eta'$ for $\eta=Y,\bar{Y},Z,\bar{Z}$. By
virtue of It\^o's formula, we have
\begin{eqnarray*}
&&\E e^{\mu t+\beta A_t}|\Delta Y_t|^2+\E\int_t^T e^{\mu s+\beta A_s}\|\Delta Z_s\|^2ds\\
&&=2\E\int_t^T e^{\mu s+\beta A_s}\Delta
Y_s\left\{f(s,{Y}_s,{Z}_s)-f(s,{Y'}_s,{Z'}_s)\right\}ds
+2\E\int_t^T e^{\mu s+\beta A_s} \Delta Y_s\left\{\phi(s,{Y}_s)-\phi(s,{Y'}_s)\right\}dA_s\\
 &&+2\E\int_t^T e^{\mu s+\beta A_s} \Delta Y_s d(\Delta K_s)+\int_t^Te^{\mu s+\beta A_s}\left\|g(s,\bar{Y}_s,\bar{Z}_s)-g(s,\bar{Y'}_s,\bar{Z'}_s)\right\|^2ds\\
 &&-\mu\E\int_t^T
e^{\mu s+\beta A_s} \left|\Delta Y_s\right|^2 ds-\beta\E\int_t^T
e^{\mu s+\beta A_s} \left|\Delta Y_s\right|^2 dA_s.
\end{eqnarray*}
From$({\bf H2})$ there exists $\alpha<\alpha'<1$ such that
\begin{eqnarray*}
\Delta Y_s\left\{f(s,{Y}_s,{Z}_s)-f(s,{Y'}_s,{Z'}_s)\right\}&\leq& \left[\frac{c}{1-\alpha'}+(1-\alpha')\right]|\Delta Y_s|^2+(1-\alpha')|\Delta Z_s|^2\\
\Delta Y_s\left\{\phi(s,{Y}_s)-\phi(s,{Y'}_s)\right\}&\leq &\beta |\Delta Y_s|^2,
\end{eqnarray*}
which together with $\displaystyle\E\int_t^T e^{\mu s+\beta A_s} \Delta Y_s
d(K_s-K'_s)\leq 0$, provide
\begin{eqnarray*}
&&\E e^{\mu t+\beta A_t}|\Delta Y_t|^2+\alpha'\E\int_t^T e^{\mu s+\beta A_s}\|\Delta Z_s\|^2ds\\
&&\leq \left(\frac{c}{1-\alpha'}+1-\alpha'-\mu\right)\E\int_t^T
e^{\mu s+\beta A_s}|\Delta Y_s|^2 ds+\beta\E\int_t^T e^{\mu s+\beta
A_s}|\Delta
Y_s|^2 dA_s\\
&&+c\E\int_t^Te^{\mu s+\beta A_s}|\Delta\bar{Y}_s|^{2}ds+\alpha\E\int_t^Te^{\mu s+\beta A_s}\|\Delta\bar{Z}_s\|^2ds
\end{eqnarray*}
Next, choosing \ $\mu$\ such that\ $\displaystyle
\mu-\frac{c}{1-\alpha'}-1+\alpha'=\frac{\alpha'c}{\alpha}$, we obtain
\begin{eqnarray*}
&&\alpha'\left[\frac{c}{\alpha}\E\int_t^T
e^{\mu s+\beta A_s} \left|\Delta Y_s\right|^2 ds+\frac{|\beta|}{\alpha'}\E\int_t^T e^{\mu s+\beta A_s}|\Delta
Y_s|^2 dA_s+\E\int_t^T e^{\mu s+\beta A_s}\|\Delta Z_s\|^2ds\right]\\
&&\leq \alpha\left(\frac{c}{\alpha}\E\int_t^T e^{\mu s+\beta
A_s}\left|\Delta \bar {Y}_s\right|^2ds +\frac{|\beta|}{\alpha'}\E\int_t^T e^{\mu
s+\beta A_s}\left|\Delta \bar {Y}_s\right|^2dA_s+\E\int_t^Te^{\mu
s+\beta A_s}\left\|\Delta\bar{Z}_s)\right\|^2ds\right).
\end{eqnarray*}
Therefore
\begin{eqnarray*}
&&\bar{c}\E\int_t^T
e^{\mu s+\beta A_s} \left|\Delta Y_s\right|^2 ds+\bar{\beta}\E\int_t^T e^{\mu s+\beta A_s}|\Delta
Y_s|^2 dA_s+\E\int_t^T e^{\mu s+\beta A_s}\|\Delta Z_s\|^2ds\\
&&\leq \frac{\alpha}{\alpha'}\left(\bar{c}\E\int_t^T e^{\mu s+\beta
A_s}\left|\Delta \bar {Y}_s\right|^2ds +\bar{\beta}\E\int_t^T e^{\mu
s+\beta A_s}\left|\Delta \bar {Y}_s\right|^2dA_s+\E\int_t^Te^{\mu
s+\beta A_s}\left\|\Delta\bar{Z}_s)\right\|^2ds\right)
\end{eqnarray*}
where  $\displaystyle \bar{c}=c/\alpha$ and $\displaystyle \bar{\beta}=|\beta|/\alpha'$.

Since $\displaystyle\frac{\alpha}{\alpha'}<1$, then $\Phi$ is a strict contraction on
$\mathcal{S}^{2}([0,T],\R)\times \mathcal{M}^{2}((0,T);\R^d)$
equipped with the norm
\begin{eqnarray*}
\|(Y,Z)\|^{2}=\bar{c}\E\int_t^T
e^{\mu s+\beta A_s} \left|Y_s\right|^2 ds+\bar{\beta}\E\int_t^T e^{\mu s+\beta A_s}|
Y_s|^2 dA_s+\E\int_t^T e^{\mu s+\beta A_s}\|Z_s\|^2ds.
\end{eqnarray*}
Then it has a unique fixed point, which is the unique solution  of BDSDE \eqref{a1}.
\end{proof}

\section {Connection to stochastic viscosity solution for reflected SPDEs with nonlinear
Neumann boundary condition} \setcounter{theorem}{0}
\setcounter{equation}{0} In this section we will investigate the
reflected generalized BDSDEs studied in the previous
 section in order to give a interpretation for the stochastic viscosity
 solution of a class of nonlinear reflected SPDEs with nonlinear Neumann boundary condition.

\subsection {Notion of stochastic viscosity solution for reflected SPDEs with nonlinear Neumann boundary condition}
With the same notations as in Section 2, let ${\bf
F}^{B}=\{\mathcal{F}_{t,T}^{B}\}_{0\leq t\leq T}$ be the filtration
generated by $B$. The set ${\mathcal{M}}^{B}_{0,T}$ denote all the ${\bf F}^{B}$-stopping
times $\tau$ such $0\leq \tau\leq T$, a.s. For generic
Euclidean spaces $E$ and $E_{1}$ we introduce the following:
\begin{enumerate}
\item The symbol $\mathcal{C}^{k,n}([0,T]\times
E; E_{1})$ stands for the space of all $E_{1}$-valued functions
defined on $[0,T]\times E$ which are $k$-times continuously
differentiable in $t$ and $n$-times continuously differentiable in
$x$, and $\mathcal{C}^{k,n}_{b}([0,T]\times E; E_{1})$ denotes the
subspace of $\mathcal{C}^{k,n}([0,T]\times E; E_{1})$ in which all
functions have uniformly bounded partial derivatives.
\item For any sub-$\sigma$-field $\mathcal{G} \subseteq
\mathcal{F}_{T}^{B}$, $\mathcal{C}^{k,n}(\mathcal{G},[0,T]\times E;
E_{1})$ (resp.\, $\mathcal{C}^{k,n}_{b}(\mathcal{G},[0,T]\times E;
E_{1})$) denotes the space of all $\mathcal{C}^{k,n}([0,T]\times E;
E_{1})$  (resp.\, $\mathcal{C}^{k,n}_{b}([0,T]\times E;E_{1})$-valued
random variable that are $\mathcal{G}\otimes\mathcal{B}([0,T]\times
E)$-measurable;
\item $\mathcal{C}^{k,n}({\bf F}^{B},[0,T]\times E; E_{1})$
(resp.$\mathcal{C}^{k,n}_{b}({\bf F}^{B},[0,T]\times E; E_{1})$) is
the space of all random fields $\phi\in
\mathcal{C}^{k,n}({\mathcal{F}}_{T},[0,T]\times E; E_{1}$ (resp.
$\mathcal{C}^{k,n}({\mathcal{F}}_{T},[0,T]\times E; E_{1})$, such
that for fixed $x\in E$ and $t\in [0,T]$, the mapping
$\displaystyle{\omega\mapsto \alpha(t,\omega,x)}$ is
${\bf F}^{B}$-progressively measurable.
\item For any sub-$\sigma$-field $\mathcal{G} \subseteq
\mathcal{F}^{B}$ and a real number $ p\geq 0$,
$L^{p}(\mathcal{G};E)$ to be all $E$-valued $\mathcal{G}$-measurable
random variable $\xi$ such that $ \E|\xi|^{p}<\infty$.
\end{enumerate}
Furthermore, regardless their dimensions we denote by
$<.,.>$ and $|.|$ the inner product and norm in $E$ and
$E_1$, respectively. For $(t,x,y)\in[0,T]\times\R^{d}\times\R$, we
denote $D_{x}=(\frac{\partial}{\partial
x_{1}},....,\frac{\partial}{\partial x_{d}}),\,\\
D_{xx}=(\partial^{2}_{x_{i}x_{j}})_{i,j=1}^{d}$,
$D_{y}=\frac{\partial}{\partial y}, \,\
D_{t}=\frac{\partial}{\partial t}$. The meaning of $D_{xy}$ and
$D_{yy}$ is then self-explanatory.

Let $\Theta$ be an open connected and smooth bounded domain of $\R^{d}\, (d\geq
1)$ such that
for a function $\psi\in\mathcal{C}^{2}_b(\R^{d}),\ \Theta$ and its
boundary  $\partial\Theta$ are characterized by
$\Theta=\{\psi>0\},\, \partial\Theta=\{\psi=0\}$ and, for any
$x\in\partial\Theta,\, \nabla\psi(x)$ is the unit normal vector
pointing towards the interior of $\Theta$.

In this section, we shall make use of the following standing assumptions:
\begin{description}
\item $({\bf A1})$ The functions $\sigma:\R^{d}\rightarrow\R^{d\times
d}$ and $b:\R^{d}\rightarrow\R^{d}$ are uniformly Lipschitz continuous,
with a common Lipschitz constant $K>0$.
\item $({\bf A2})$ The functions $f:\Omega\times[0,T]\times\overline{\Theta}\times\R\times\R^{d}\rightarrow\R$ and $\phi:\Omega\times[0,T]\times\overline{\Theta}\times\R\rightarrow\R$ are continuous random field such
that for fixed $(x,y,z)$ and $f(\cdot,x,y,z),\;\phi(\cdot,x,y)$ and $h(\cdot,x)$ are ${\bf F}^{B}$-progressively measurable; and there exists $K>0$, such that for $\P$-a.e $\omega$,
\begin{description}
\item $(i)\;|f(\omega,t,x,y,z)|\leq K(1+|x|+|y|+\|z\|)$,
\item $(ii)\;|f(\omega,t,x,y,z)-f(\omega,t',x',y',z')|\leq c(|t-t'|+|x-x'|+|y-y'|+\|z-z'\|)$,
\item $(iii)\;|\phi(\omega,t,x,y)|\leq K(1+|x|+|y|)$,
\item $(iv)\;\langle y-y',\phi(\omega,t,x,y)-\phi(\omega,t,x,y')\rangle\leq \beta|y-y'|^2$
\item $(v)\;|\phi(\omega,t,x,y)-\phi(\omega,t,x,y')|\leq K(|x-x'|+|y-y'|)$
\end{description}
\item $({\bf A3})$ The function $l:\R^n\rightarrow\R$ is continuous, such that for some constants $K,\;p>0$
\begin{description}
\item $|l(x)|\leq K(1 + |x|^p),\; x\in\R^n$.
\end{description}
\item $({\bf A4})$ The function $h:\Omega\times[0,T]\times\overline{\Theta}\rightarrow\R$ is continuous random field such that for fixed $x,\; h(\cdot,x)$ is ${\bf F}^{B}$-progressively measurable; and there exists $K>0$, such that for $\P$-a.e $\omega$,
\begin{description}
\item $(i)\;|h(\omega,t,x)|\leq K(1+|x|)$,
\item $(ii)\; h(\omega,T,x)\leq l(x)$.
\end{description}
\item $({\bf A5})$ The function $g\in{\mathcal{C}}_{b}^{0,2,3}([0,T]\times\overline{\Theta}\times\R;\R)$.
\end{description}
Let us consider the related obstacle problem for SPDE with nonlinear
Neumann boundary condition:
\begin{eqnarray*}
\mathcal{OP}^{(f,\phi,g,h,l)}\left\{
\begin{array}{l}
(i)\,\displaystyle \min\left\{u(t,x)-h(t,x),\; -\frac{\partial
u(t,x)}{\partial t}-[
Lu(t,x)+f(t,x,u(t,x),\sigma^{*}(x)D_{x}u(t,x))]\right.\\\\
\left.\,\,\,\,\,\,\,\,\,\,\,\,\,\,\,\,\,\ -
g(t,x,u(t,x))\dot{B}_{s}\right\}=0,\,\,\
(t,x)\in[0,T]\times\Theta\\\\
(ii)\,\displaystyle\frac{\partial u}{\partial
n}(t,x)+\phi(t,x,u(t,x))=0,\,\,\ (t,x)\in[0,T]\times\partial\Theta, \\\\
(iii)\,u(T,x)=l(x),\,\,\,\,\,\,\ x\in\overline{\Theta}
\end{array}\right.
\end{eqnarray*}
where
\begin{eqnarray*}
L=\frac{1}{2}\sum_{i,j=1}^{d}(\sigma(x)\sigma^{*}(x))_{i,j}
\frac{\partial^{2}}{\partial x_{i}\partial
x_{j}}+\sum^{d}_{i=1}b_{i}(x)\frac{\partial}{\partial x_{i}},\quad
\forall\, x\in\Theta,
\end{eqnarray*}
and
\begin{eqnarray*}
\frac{\partial}{\partial n}=\sum_{i=1}^{d}\frac{\partial
\psi}{\partial x_{i}}(x)\frac{\partial}{\partial x_{i}},\quad
\forall\, x\in\partial\Theta.
\end{eqnarray*}

\begin{remark}
In the previous definition, the stochastic ``variational inequality" is defined formally since it involves a quantitative comparison between a random field and its stochastic differential. Therefore it does actually make sense as follows (see \cite{ML}):
there exists a regular random measure $\nu$ such that $(i)$ becomes
\begin{eqnarray*}
\left\{
\begin{array}{l}
(iv)\,u(t,x)\geq h(t,x),\;\;\;d\P\otimes dt\otimes dx-a.e.,\\\\
(v)\,-\frac{\partial u}{\partial t}(t,x)-[Lu(t,x)+f(t,x,u(t,x),\sigma^{*}(x)\nabla u(t,x))]-g(t,x,u(t,x))\dot{B}_{s}=-\nu(dt,dx),\\\\ a.s.,(t,x)\in[0,T]\times\Theta,\\\\
(vi)\,\nu(u>h)=0,\,\,a.s.
\end{array}\right.
\end{eqnarray*}
\end{remark}
Our next goal is to define the notion of stochastic viscosity solution to
$\mathcal{OP}^{(f,\phi,g,h,l)}$. In fact, we recall some of
the notations appear in \cite{BM1}. Let $\eta\in\mathcal{C}({\bf
F}^{B},[0,T]\times\R^{d}\times\R)$ be the solution of equation
\begin{eqnarray*}
\eta(t,x,y)&=&y+\int_t^T\langle g(s,x,\eta(s,x,y)),
\circ{\overleftarrow{dB}}_s\rangle,
\end{eqnarray*}
where $\circ{\overleftarrow{dB}}$ is the Stratonowich backward stochastic integral which respect the Brownian $B$. We have equivalence with Itô backward stochastic integral which respect the Brownian $B$ as follows:
\begin{eqnarray*}
\int_t^T\langle g(s,x,\eta(s,x,y)),
\circ{\overleftarrow{dB}}_s\rangle&=&\frac{1}{2}\int_t^T\langle g,D_yg\rangle(s,x,\eta(s,x,y)ds
+\int_t^T\langle g(s,x,\eta(s,x,y)),\overleftarrow{dB}_s\rangle.\label{c'1}
\end{eqnarray*}
By $({\bf A5})$ and for all $(t,x)\in [0,T]\times\R^d$, the mapping $y\mapsto\eta(s,x,y)$ defines a diffeomorphism
almost surely. Hence if we denote by $\varepsilon(s,x,y)$ its
$y$-inverse, one can show that (cf. Buckdahn and Ma \cite{BM1})
\begin{eqnarray}
\varepsilon(t,x,y)=y-\int_t^T\langle D_{y}\varepsilon(s,x,y)
g(s,x,y),\circ {\overleftarrow{dB}}_{s}\rangle. \label{c2}
\end{eqnarray}
To simplify the notation in the sequel we denote
\begin{eqnarray*}
&&A_{f,g}(\varphi(t,x))=-
L\varphi(t,x)-f(t,x,\varphi(t,x),\sigma^{*}D_{x}\varphi(t,x))+\frac{1}{2}
(g,D_{y}g) (t,x,\varphi(t,x))\\
&&\mbox{and}\, \Psi(t,x)=\eta(t,x,\varphi(t,x)).
\end{eqnarray*}
\begin{definition}\label{D:defvisco}
A random field $u \in \mathcal{C}\left(\mathbf{F}^B, [0,T]\times
\overline{\Theta}\right)$ is called a stochastic viscosity
subsolution of the stochastic obstacle problem
$\mathcal{OP}^{(f,\phi,g,h,l)}$ if $u\left(T,x\right)\leq
l\left(x\right)$, for all $x\in \overline{\Theta}$, and if for any
stopping time $\tau \in \mathcal{M}_{0,T}^B$, any state variable
$\xi\in L^0\left(\mathcal{F}_{\tau}^B, \overline{\Theta}\right)$, and any
random field $\varphi\in
\mathcal{C}^{1,2}\left(\mathcal{F}_{\tau}^B,\ [0,T]\times
\mathbb{R}^d\right)$ such that for $\mathbb{P}$-almost
all $\omega\in\left\{0<\tau<T\right\}$,
$$u\left(t,\omega,x\right)-\Psi\left(t,\omega,x\right)
 \leq 0=u\left(\tau(\omega),\xi(\omega)\right)-\Psi
\left(\tau(\omega),\xi(\omega)\right)$$ for all
$\left(t,x\right)$ in some neighborhood
$\mathcal{V}\left(\omega,\tau\left(\omega\right),\xi\left(\omega\right)\right)$
of $\left(\tau\left(\omega\right),\xi\left(\omega\right)\right)$,
it holds:
\begin{itemize}
\item[(a)] on the event $\left\{0<\tau<T\right\}\cap\left\{\xi\in
\Theta\right\}$
\begin{eqnarray}\label{def1}
\min\left\{u(\tau,\xi)-h(\tau,\xi),\mathrm{A}_{f,
g}\left(\Psi\left(\tau,\xi\right)\right)- D_y\Psi
\left(\tau,\xi\right)D_t \varphi\left(\tau,\xi\right)\right\}\leq 0
\end{eqnarray}
holds, $\mathbb{P}$-almost surely;
\item[(b)] on the event $\left\{0<\tau<T\right\}\cap\left\{\xi\in
\partial \Theta\right\}$ the inequality
\begin{eqnarray} \min & \left[\min\left\{u(\tau,\xi)-h(\tau,\xi),\mathrm{A}_{f,
g}\left(\Psi\left(\tau,\xi\right)\right)- D_y \Psi
\left(\tau,\xi\right) D_t
\varphi\left(\tau,\xi\right)\right\}\,,\right.\nonumber\\
& \left.-\frac{\displaystyle{\partial \Psi}}{\displaystyle{\partial
n}}\left(\tau,\xi\right)-\phi\left(\tau,\xi,\Psi\left(\tau,\xi\right)\right)
\right] \leq 0 \label{viscosity01}
\end{eqnarray}
holds, $\mathbb{P}$-almost surely.
\end{itemize}
A random field $u \in \mathcal{C}\left(\mathbf{F}^B, [0,T]\times
\overline{\Theta}\right)$ is called a stochastic viscosity
supersolution of the stochastic obstacle problem
$\mathcal{OP}^{(f,\phi,g,h,l)}$ if $u\left(T,x\right)\geq
l\left(x\right)$, for all $x\in \overline{\Theta}$, and if for any
stopping time $\tau\in\mathcal{M}_{0,T}^B$, any state variable
$\xi\in L^0\left(\mathcal{F}_{\tau}^B,\overline{\Theta}\right)$, and any random
field $\varphi\in \mathcal{C}^{1,2}\left(\mathcal{F}_{\tau}^B,\
[0,T]\times \mathbb{R}^d\right)$ such that for
$\mathbb{P}$-almost all $\omega\in\left\{0<\tau<T\right\}$
$$u\left(t,\omega,x\right)-\Psi \left(t,\omega,x\right)
 \geq 0=u\left(\tau(\omega),\xi(\omega)\right)-\Psi
\left(\tau(\omega),\xi(\omega)\right)$$ for all
$\left(t,x\right)$ in some neighborhood
$\mathcal{V}\left(\omega,\tau\left(\omega\right),\xi\left(\omega\right)\right)$
of $\left(\tau\left(\omega\right),\xi\left(\omega\right)\right)$, it holds:
\begin{itemize}
\item[(a)] on the event $\left\{0<\tau<T\right\}\cap\left\{\xi\in
\Theta\right\}$
\begin{equation}\label{E:def12}
\min\left\{u(\tau,\xi)-h(\tau,\xi),\mathrm{A}_{f,
g}\left(\Psi\left(\tau,\xi\right)\right) - D_y\Psi
\left(\tau,\xi\right)D_t \varphi\left(\tau,\xi\right)\right\}\geq 0
\end{equation}
holds, $\mathbb{P}$-almost surely;
\item[(b)] on the event $\left\{0<\tau<T\right\}\cap\left\{\xi\in
\partial \Theta\right\}$
\begin{align} \max & \left[\min\left\{u(\tau,\xi)-h(\tau,\xi),\mathrm{A}_{f,
g}\left(\Psi\left(\tau,\xi\right)\right)- D_y \Psi
\left(\tau,\xi\right) D_t
\varphi\left(\tau,\xi\right)\right\}\,,\right.\nonumber\\
& \left.-\frac{\displaystyle{\partial \Psi}}{\displaystyle{\partial
n}}\left(\tau,\xi\right)-\phi\left(\tau,\xi,\Psi\left(\tau,\xi\right)\right)
\right] \geq 0 \label{E:viscosity02}
\end{align}
 holds, $\mathbb{P}$-almost surely.
 \end{itemize}

Finally, a random field $u \in \mathcal{C}\left(\mathbf{F}^B,
[0,T]\times \overline{\Theta}\right)$ is called a stochastic
viscosity solution of the stochastic obstacle problem
$\mathcal{OP}^{(f,\phi,g,h,l)}$ if it is both a stochastic
viscosity subsolution and a stochastic viscosity supersolution.
\end{definition}
\begin{remark}
Observe that if $f,\,\phi$ are deterministic and $g\equiv 0$, the flow $\eta$ becomes $\eta(t,x,y)=y$ and $\Psi(t,x)=\varphi(t,x),\, \forall\ (t,x,y)\in[0,T]\times\R^d\times\R$. Thus,
definition\, $\ref{D:defvisco}$ coincides with the definition of
(deterministic) viscosity solution of PDE
$\mathcal{OP}^{(f,\phi,0,h,l)}$ given in \cite{Ral} for each fixed $\omega\in\{0<\tau<T\}$, modulo the
${\bf F}^B$-measurability requirement of the test function $\varphi$.
\end{remark}
Now let us recall a notion of random viscosity solution which will be a bridge linking the
stochastic viscosity solution and its deterministic counterpart.
\begin{definition}
A random field $u\in C({\bf F}^B, [0,T]\times\R^n)$ is called an $\omega$-wise viscosity solution if
for $\P$-almost all $\omega\in \Omega,\;  u(\omega,\cdot,\cdot)$ is a (deterministic) viscosity solution of $\mathcal{OP}^{(f,\phi,0,h,l)}$.
\end{definition}

Next we introduce the Doss-Sussman transformation. It enables us to convert an reflected SPDE of
the form $\mathcal{OP}^{(f,\phi,g,h,l)}$ to an classical partial differential equation of the form $\mathcal{OP}^{(\widetilde{f},0,\widetilde{\phi}, \widetilde{h})}$ where $\widetilde{f}$, $\widetilde{\phi}$ and $\widetilde{h}$ are certain well-defined random fields, which are defined in terms of $f,\,\phi$ and $h$.
\begin{proposition}
Assume $({\bf A1})$-$({\bf A5})$. A random field $u$ is a stochastic viscosity solution to the
$\mathcal{OP}^{(f,\phi,g,h,l)}$ if and only if $v(\cdot, \cdot) = \varepsilon(\cdot, \cdot,u(\cdot, \cdot))$ is a stochastic viscosity solution to the SPDE
$\mathcal{OP}^{(\widetilde{f},0,\widetilde{\phi}, \widetilde{h},l)}$, where $(\widetilde{f},\widetilde{\phi}, \widetilde{h})$ are three coefficients that will be made precise later (see \eqref{Doss1},\eqref{Doss2} and \eqref{Doss3}).
\end{proposition}
\begin{proof}
We shall only prove that if $u\in\mathcal{C}(\mathbf{F}^{B},[0,T]\times\overline{\Theta})$ is a stochastic viscosity
sub-(resp. super-)solution to SPDE $\mathcal{OP}^{(f,\phi,g,h,l)}$, then $v(\cdot, \cdot) = \varepsilon(\cdot, \cdot,u(\cdot, \cdot))$
belongs to $\mathcal{C}(\mathbf{F}^{B},[0,T]\times\overline{\Theta})$, and it is a stochastic viscosity sub-(resp. super-) solution to SPDE $\mathcal{OP}^{(\widetilde{f},\widetilde{\phi},0,\widetilde{h},l)}$. The converse part of the proposition can be proved in a very similar way. We shall only discuss for the stochastic subsolution case, as the supersolution part can be proved similarly. Therefore, let us assume that $u\in\mathcal{C}(\mathbf{F}^{B},[0,T]\times\overline{\Theta})$ is the is a
stochastic viscosity subsolution of the SPDE $\mathcal{OP}^{(f,\phi,g,h,l)}$. It then follows that $v(\cdot, \cdot) = \varepsilon(\cdot, \cdot,u(\cdot, \cdot))$
belongs to $\mathcal{C}(\mathbf{F}^{B},[0,T]\times\overline{\Theta})$. Let now show that $v$ is a stochastic viscosity subsolution of the SPDE $\mathcal{OP}^{(\widetilde{f},0,\widetilde{\phi}, \widetilde{h},l)}$. Firstly, since $y\mapsto\varepsilon(\cdot,\cdot,y)$ is increasing, $v(t,x)\geq \widetilde{h}(t,x), \forall\, (t,x)\in [0,T]\times\overline{\Theta}$, where
\begin{eqnarray}
\widetilde{h}(t,x)=\varepsilon(t,x,h(t,x)).\label{Doss3}
\end{eqnarray}
Following, there exist $(\tau,\xi)\in\mathcal{M}^{B}_{0,T}\times L^{0}({\mathcal{F}}^{B}_{\tau};\overline{\Theta})$ satisfy
\begin{eqnarray}
\P(v(\tau,\xi)>\widetilde{h}(\tau,\xi),\, 0<\tau<T)>0;\label{visco1}
\end{eqnarray}
and $\varphi\in\mathcal{C}^{1,2}\left(\mathcal{F}_{\tau}^B,\ [0,T]\times\overline{\Theta}\right)$ such that for $\P$-almost all $\displaystyle{\omega\in\{0<\tau<T,v(\tau,\xi)>\widetilde{h}(\tau,\xi)\}}$, the inequality
\begin{eqnarray}
u(\omega,t,x)-\Psi(\omega,t,x)\leq 0=u(\omega,\tau(\omega),\xi(\omega))-\Psi(\omega,\tau(\omega),\xi(\omega))\label{V}
\end{eqnarray}
holds for all $(t, x)$ in some neighborhood $\mathcal{V}(\omega,\tau(\omega),\xi(\tau))$ of $(\omega,\tau(\omega))$. Next putting
$\Psi(t,x) = \eta(t,x,\varphi(t,x))$ and since the mapping $y\mapsto\eta(t, x, y)$ is strictly increasing, for all $(t, x)\in\mathcal{V}(\tau,\xi)$ we get that
\begin{eqnarray*}
u(t, x)-\Psi(t,x) &=& \eta(t,x,v(t,x))-\eta(t,x,\varphi(t,x))\\
&\leq &0 = \eta(\tau,v(\tau,\xi))-\eta(\tau,\varphi(\tau\xi))\\
&=& u(\tau, \xi)-\Psi(\tau, \xi)
\end{eqnarray*}
holds $\P$-almost surely on $\displaystyle{\omega\in\{0<\tau<T,v(\tau,\xi)>\widetilde{h}(\tau,\xi)\}}$.
According to \eqref{visco1} and recall again the strictly increasing of the mapping $y\mapsto\eta(t, x, y)$, we have\newline $\displaystyle{\{v(\tau,\xi)>\widetilde{h}(\tau,\xi)\}=\{u(\tau,\xi)>h(\tau, \xi)\}}$.
Moreover, since $u$ is a stochastic viscosity subsolution of the SPDE $\mathcal{OP}^{(f,\phi,g,h,l)}$, we obtain:
\begin{itemize}
\item[(a)] on the event $\left\{0<\tau<T\right\}\cap\{u(\tau,\xi)>h(\tau, \xi)\}\cap\left\{\xi\in
\Theta\right\}$
\begin{eqnarray*}\label{E:def1}
\mathrm{A}_{f,g}\left(\Psi\left(\tau,\xi\right)\right)- D_y\Psi
\left(\tau,\xi\right)D_t \varphi\left(\tau,\xi\right)\leq 0
\end{eqnarray*}
holds, $\mathbb{P}$-almost surely;
\item[(b)] on the event $\left\{0<\tau<T\right\}\cap\{u(\tau,\xi)>h(\tau, \xi)\}\cap\left\{\xi\in
\partial \Theta\right\}$ the inequality
\begin{eqnarray*} \min\left[\mathrm{A}_{f,
g}\left(\Psi\left(\tau,\xi\right)\right)- D_y \Psi
\left(\tau,\xi\right) D_t
\varphi\left(\tau,\xi\right),-\frac{\displaystyle{\partial \Psi}}{\displaystyle{\partial
n}}\left(\tau,\xi\right)-\phi\left(\tau,\xi,\Psi\left(\tau,\xi\right)\right)
\right] \leq 0 \label{E:viscosity01}
\end{eqnarray*}
holds, $\mathbb{P}$-almost surely.
\end{itemize}
By the similarly calculation used in \cite{Bal}, we have:
\begin{itemize}
\item[(a)] on the event $\left\{0<\tau<T\right\}\cap\{v(\tau,\xi)>\widetilde{h}(\tau, \xi)\}\cap\left\{\xi\in
\Theta\right\}$ the inequality
\begin{eqnarray}\label{E:def11}
\mathrm{A}_{\widetilde{f},0}\left(\varphi\left(\tau,\xi\right)\right)-D_t \varphi\left(\tau,\xi\right)\leq 0
\end{eqnarray}
holds, $\mathbb{P}$-almost surely;
\item[(b)] on the event $\left\{0<\tau<T\right\}\cap\{v(\tau,\xi)>\widetilde{h}(\tau, \xi)\}\cap\left\{\xi\in
\partial \Theta\right\}$ the inequality
\begin{eqnarray} \min\left[\mathrm{A}_{\widetilde{f},
0}\left(\varphi\left(\tau,\xi\right)\right)-D_t
\varphi\left(\tau,\xi\right),-\frac{\displaystyle{\partial \varphi}}{\displaystyle{\partial
n}}\left(\tau,\xi\right)-\widetilde{\phi}\left(\tau,\xi,\varphi\left(\tau,\xi\right)\right)
\right] \leq 0 \label{E:viscosity011}
\end{eqnarray}
holds, $\mathbb{P}$-almost surely.
\end{itemize}
where
\begin{eqnarray}
\widetilde{f}(t,x,y,z)&=&\frac{1}{D_y\eta(t,x,y)}
\left[f\left(t,x,\eta(t,x,y),\sigma(x)^{*}D_x\eta(t,x,y)+D_y\eta(t,x,y)z\right)\right.\nonumber\\
&&\left.-\frac{1}{2}gD_yg(t,x,\eta(t,x,y))+L_x\eta(t,x,y)+\langle\sigma(x)^{*}D_{xy}\eta(t,x,y),z\rangle\right.\nonumber\\
&&\left.+\frac{1}{2}D_{yy}\eta(t,x,y)|z|^2\right]\label{Doss1}
\end{eqnarray}
and
\begin{eqnarray}
\widetilde{\phi}(t,x,y)=\frac{1}{D_y\eta(t,x,y)}\left[h(t,x,\eta(t,x,y))+D_x\eta(t,x,y)\nabla\psi(x)\right].\label{Doss2}
\end{eqnarray}
Combining inequality \eqref{E:def11} and \eqref{E:viscosity011}, we obtain that the random field $v$ is a stochastic
viscosity subsolution of the SPDE $\mathcal{OP}^{(\widetilde{f},\widetilde{\phi},0,\widetilde{h},l)}$, which ends the proof of Proposition 3.1.
\end{proof}
\subsection {Existence of stochastic viscosity solutions for SPDE
with nonlinear Neumann boundary condition} The main objective of
this subsection is to give a link between the stochastic obstacle problem
$\mathcal{OP}^{(f,\phi,g,h,l)}$ and the reflected generalized
BDSDE $(\ref{a1})$ introduced in Section 1. We consider
\allowdisplaybreaks\begin{align}
s\mapsto &A_{s}^{t,x} \,\,\,\hbox{is increasing}\nonumber\\
X_s^{t,x} &= x+\int^{s\vee t}_t b\left(X_r^{t,x}\right) dr+\int^{s\vee t}_t
\sigma\left(X_r^{t,x}\right) d{W}_r+\int^{s\vee t}_t \nabla \psi
\left(X_r^{t,x}\right) dA_r^{t,x}, \quad
\forall\, s\in [t,T]\,,\nonumber\\
A^{t,x}_{s}&=\int^{s\vee t}_t I_{\left\{ X^{t,x}_{r}\in\partial\Theta
\right\}}\, dA^{t,x}_{r}.\label{rSDE}
\end{align}
It is clear (see \cite{LZ}) that under
conditions $({\bf A1})$ on the coefficients $b$ and $\sigma$,
$(\ref{rSDE})$ has a unique strong ${\bf F}^{W}$-adapted solution.

Using the similar arguments as in Pardoux and Zhang \cite{PZ}( Propositions 3.1 and 3.2),
or Slomi\`nski \cite{Sl}, we can provide the following regularity results.
\begin{proposition}\label{P:continuity00}
There exists a constant $ C>0 $ such that for all for all $0\leq
t<t'\leq T$ and $x,\,x'\in \overline{\Theta}$,\  the following
inequalities hold: for any $p>4$
$$\mathbb{E}\left[\sup_{0\leq s\leq
T}\left|X^{t,x}_{s}-X^{t',x'}_{s}\right|^p\right] \leq
C\left\{ |t'-t|^{p/2}+|x-x'|^p\right\}$$
and
$$\mathbb{E}\left[\sup_{0\leq s\leq
T}\left|A_s^{t,x}-A_{s}^{t',x'}\right|^p\right]\leq C\left\{
|t'-t|^{p/2}+|x-x'|^p\right\}.$$
Moreover, for all $p\geq 1$,
there exists a constant $C_p$ such that for all
$(t,x)\in[0,T]\times \overline{\Theta}$,
\[\mathbb{E}\left(\left|A_s^{t,x}\right|^p \right) \leq C_p(1+t^p)\] and
for each $\mu$, $t<s<T$, there exists a constant $C(\mu,t)$ such
that for all $x\in\overline{\Theta}$,
$$\mathbb{E}\left(\displaystyle e^{\mu A_{s}^{t,x}}\right) \leq C(\mu,t).$$
\end{proposition}
We consider also the following reflected
generalized BDSDE: for $(t,x)\in[0,T]\times\overline{\Theta}$
\begin{align}\label{E:back}
\left\{
\begin{aligned}
&Y_s^{t,x}  =   l\left(X^{t,x}_T\right)+\int_{s\vee t}^T
f\left(r,X^{t,x}_r ,Y_r^{t,x},Z_r^{t,x}\right) dr+\int_{s\vee t}^T
g\left(r,X^{t,x}_r ,Y_r^{t,x}\right) \overleftarrow{dB}_r
\\
& \qquad\qquad +\int_{s\vee t}^T \phi\left(r,X_r^{t,x}, Y_r^{t,x}\right)
dA_r^{t,x}+K^{t,x}_T-K_{s\vee t}^{t,x}-\int_{s\vee t}^TZ_r^{t,x}{dW_r},\\
& Y^{t,x}_s\geq h(s,X_s^{t,x})\, \mbox{such that}\ \displaystyle\int_{s\vee t}^T\left(Y^{t,x}_r-h(r,X_r^{t,x})\right)dK^{t,x}_r=0.
\end{aligned}
\right.
\end{align}
where the coefficients $l$, $f$, $g$, $\phi$ and $h$ satisfy the
hypotheses $({\bf A2})$-$({\bf A5})$.
The following regularity result generalizes the Kolmogorov continuity criterion to BDSDEs:
\begin{proposition}\label{cont}
Let the ordered triplet $(Y^{t,x}_s,Z^{t,x}_s,K^{t,x}_s) $ be a
solution of the BDSDE $(\ref{E:back})$. Then the random field
$(s,t, x)\mapsto Y^{t,x}_s$ is almost surely continuous on $[0,T]\times[0,T]\times \overline{\Theta}$.
\end{proposition}
\begin{proof}
In the same way as in the proof of Lemma 2.1, we have, for $t,t'\in [0,T],\; x,\,x'\in\Theta$ and $p> 4$,
\begin{eqnarray*}
&&\E\left(\sup_{0\leq s\leq T}\left\vert
Y^{t,x}_{s}-Y^{t',x'}_{s}\right\vert^p\right)+\E|K^{t,x}_T-K^{t',x'}_T|^p+\E\left(\int_0^T|Z^{t,x}_r-Z^{t',x'}_r|^{2}dr\right)^{p/2}
+\E\left(\int_0^T|Y^{t,x}_r-Y^{t',x'}_r|^{p}dA^{t,x}_r\right)\\
&&\leq C\left[\E\left(\sup_{0\leq s\leq T}\left|X^{t,x}_s-X^{t',x'}_s\right|^p\right)+ \E\left(\int_{0}^T[{\bf 1}_{[t,T]}f\left(r,X^{t,x}_r,Y^{t,x}_r,Z^{t,x}_r\right)
-{\bf 1}_{[t',T]}f(r,X^{t,x}_r,Y^{t,x}_r,Z^{t,x})]dr\right)\right.\\
&&\left.+\E\left(\int_{t\wedge t'}^{t\vee t'}\left|\phi\left(r,X^{t,x}_r,Y^{t,x}_r\right)\right|^pdA^{t,x}_r\right)+ \E\left(\int_{t\wedge t'}^{t\vee t'}\left|h\left(r,X^{t,x}_r\right)\right|[dK_r^{t,x}+dK_r^{t',x'}]\right)+\E\left(\int_0^T|X^{t,x}_r-X^{t',x'}_r|^pdA_r^{t,x}\right)\right.\\
&&\left.+\left(\E\sup_{0\leq s\leq T}|A^{t,x}_s-A_s^{t',x'}|^p\right)^{1/2}\right].
\end{eqnarray*}
Next, using Proposition 3.2 one can derive
\begin{eqnarray*}
\E\left(\sup_{0\leq s\leq T}\left\vert
Y^{t,x}_{s}-Y^{t',x'}_{s}\right\vert^p\right)\leq C(|t-t'|^{p/2}+|x-x'|^{p}+|t-t'|^{p/4}+|x-x'|^{p/2}).
\end{eqnarray*}
Therefore, il suffice to choose $p=\gamma$ convenably to get
\begin{eqnarray*}
\E\left(\sup_{0\leq s\leq T}\left\vert
Y^{t,x}_{s}-Y^{t',x'}_{s}\right\vert^{\gamma}\right)\leq C(|t-t'|^{1+\beta}+|x-x'|^{d+\delta}).
\end{eqnarray*}
We conclude from the last estimate, using Kolmogorov's lemma, that $\displaystyle{\{Y_s^{t,x}, s,t\in[0,T], x\in\overline{\Theta}\}}$
has an a.s. continuous version.
\end{proof}
Now, for each $(t,x)\in[0,T]\times\overline{\Theta},\; n\geq 1$, we consider the following BDSDE,
\begin{eqnarray}
^{n}Y^{t,x}_{s}&=&l(X_{T}^{t,x})+\int_{s}^Tf_n(r,X_{r}^{t,x},{}^{n}Y^{t,x}_{r},
{}^{n}Z^{t,x}_{r})dr\nonumber\\
&&+\int_{s}^T\phi(r,X_{r}^{t,x},^{n}Y^{t,x}_{r})dA_r^{t,x}+
\int_{s}^Tg(r,X_{r}^{t,x},^{n}Y^{t,x}_{r})\overleftarrow{dB}_{r}-\int_{s}^T
{}^{n}Z^{t,x}_{r}dW_{r}, \label{c4}
\end{eqnarray}
where $\displaystyle{f_{n}(t,x,y,z)=f(t,x,y,z)+n(y-h(t,x))^{-}}$.

Let $\{^{n}Y^{t,x}_{s},{}^{n}Z^{t,x}_{s},\,\ t\leq s\leq T\}$ denotes the
solution of  BDSDE \eqref{c4} and define $\displaystyle{u^{n}(t,x)={}^{n}Y^{t,x}_{t}}$. It is shown in Boufoussi et  al. \cite{Bal} that the function
$\displaystyle{v^n(t, x) = \varepsilon(t, x, u^n(t, x))}$ is an $\omega$-wise viscosity solution to the following SPDE
\begin{eqnarray*}
\left\{
\begin{array}{l}
(i)\,\displaystyle \frac{\partial
u^n(t,x)}{\partial t}+
Lu^n(t,x)+\widetilde{f}_n(t,x,u^n(t,x),\sigma^{*}(x)D_{x}u^n(t,x))=0,\,\,\
(t,x)\in[0,T]\times\Theta\\\\
(ii)\,\displaystyle\frac{\partial u^n}{\partial
n}(t,x)+\widetilde{\phi}(t,x,u^n(t,x))=0,\,\,\ (t,x)\in[0,T]\times\partial\Theta, \\\\
(iii)\,u^n(T,x)=l(x),\,\,\,\,\,\,\ x\in\overline{\Theta},
\end{array}\right.
\end{eqnarray*}
where $\displaystyle{\widetilde{f}_n(t,x,y,z)=\widetilde{f}(t,x,y,z)+\frac{1}{D_y\eta(t,x,y)}n(y-\widetilde{h}(t,x))^{-}}$.

Let us define for $(t, x)\in[0, T ]\time\Theta,\, u(t, \omega, x)= Y^{t,x}_t$ and $v(t, x) =\varepsilon(t, x, u(t, x))$. It follows from penalization argument, that (along a subsequence)
\begin{eqnarray*}
|v^n(\tau,\xi)-v(\tau,\xi)|\rightarrow 0,\;\; a.s.
\end{eqnarray*}
as $n$ goes to infinity.

Our main result in this section is the following:
\begin{theorem}
Let assumptions $({\bf A1})$-$({\bf A5})$ be satisfied. Then the function $u(t, x)$ defined above is a stochastic viscosity
solution of obstacle problem $\mathcal{OP}^{(f,\phi,g,h,l)}$.
\end{theorem}
\begin{proof}
By the definition of $u$ it is easy to see that $u(T, x) = l(x)$. Now, since $Y^{t,x}_s$ is $\mathcal{F}^{W}_{t,s}\otimes\mathcal{F}^{B}_{s,T}$-measurable,
it follows that $Y^{t,x}_{t}$ is $\mathcal{F}^{B}_{t,T}$-measurable. Consequently, $u(t, x)$ is $\mathcal{F}^{B}_{t,T}$-measurable and
so it is independent of $\omega'\in\Omega'$. Therefore, combining this result with Proposition \ref{cont}, we obtain $u\in C({\bf F}^{B};[0,T]\times\overline{\Theta})$.
On the other hand, it follows from it definition that for all $(\tau,\xi)\in\mathcal{M}^{B}_{0,T}\times L^{0}({\mathcal{F}}^{B}_{\tau};\overline{\Theta})$,
\begin{eqnarray*}
u(\tau(\omega),\xi(\omega))=Y^{\tau(\omega),\xi(\omega)}_{\tau(\omega)}\geq h(\tau(\omega),\xi(\omega)),\;\;\; \P\mbox{-a.s.}
\end{eqnarray*}
Thus it remains to show that $u$ satisfies \eqref{def1}-\eqref{viscosity01} and \eqref{E:def12}-\eqref{E:viscosity02}. Using Proposition 3.1, it
suffices to prove that $v$ satisfies \eqref{E:def11} and \eqref{E:viscosity011}. To this end, let $\omega\in\Omega$ be fixed such that
\begin{eqnarray}
|v^n(\omega,t,x)- v(\omega,t,x)|\rightarrow 0\,\,\ \mbox{as}\;\;\; n\rightarrow \infty,\label{cv}
\end{eqnarray}
and consider $(\tau,\xi,\varphi)\in\mathcal{M}_{0,T}^B\times L^0\left(\mathcal{F}_{\tau}^B,\overline{\Theta}\right)\times\mathcal{C}^{1,2}\left(\mathcal{F}_{\tau}^B,\
[0,T]\times \overline{\Theta}\right)$ verify, for such fixed $\omega,\:\; 0<\tau(\omega)<T,\; v(\omega,\tau(\omega),\xi(\omega))>\widetilde{h}(\omega,\tau(\omega),\xi(\omega))$ and the inequality
\begin{eqnarray}
v(\omega,t,x)-\varphi(\omega,t,x)< 0=v(\omega,\tau(\omega),\xi(\omega))-\varphi(\omega,\tau(\omega),\xi(\omega))\label{V'}
\end{eqnarray}
for all $(t,x)$ in some neighborhood
$\mathcal{V}\left(\omega,\tau\left(\omega\right),\xi\left(\omega\right)\right)$ of $\left(\tau\left(\omega\right),\xi\left(\omega\right)\right)$.
Then there exists sequence $(\tau_{n}(\omega),\xi_{n}(\omega),\varphi_n(\omega))_{n\geq 1}\in [0,T]\times\overline{\Theta}\times\mathcal{C}^{1,2}\left([0,T]\times\overline{\Theta}\right)$  satisfy
\begin{eqnarray*}
\tau_{n}(\omega)&\rightarrow& \tau(\omega),\\
\xi_{n}(\omega)&\rightarrow&\xi(\omega),\\
\varphi_n(\omega)&\rightarrow&\varphi(\omega),
\end{eqnarray*}
such that the inequality
\begin{eqnarray*}
v^{n}(\omega,t,x)-\varphi_n(\omega,t,x)<
0=v^{n}(\omega,\tau_{n}(\omega),\xi_{n}(\omega))-\varphi_n(\omega,\tau_{n}(\omega),\xi_{n}(\omega))\label{V1}
\end{eqnarray*}
holds for all $(t,x)$ in some neighborhood $\mathcal{V}\left(\tau_{n}\left(\omega\right),\xi_{n}\left(\omega\right)\right)\subset
\mathcal{V}\left(\tau\left(\omega\right),\xi\left(\omega\right)\right)$ and a suitable subsequence of $(v^n)_{n\geq1}$. Using the fact that $v^n(\omega,\cdot,\cdot)$ is a (deterministic) viscosity solution of the PDE $(\widetilde{f}_n(\omega,\cdot), \widetilde{\phi}(\omega,\cdot),0,l)$, we obtain:
\begin{itemize}
\item[(a)] if $\xi_n(\omega)\in\Theta$ the inequality
\begin{equation*}
\mathrm{A}_{\widetilde{f}_{n},0}\left(\varphi_n\left(\omega,\tau_{n}(\omega),\xi_{n}(\omega)\right)\right)-D_t
\varphi_n\left(\omega,\tau_{n}(\omega),\xi_{n}(\omega)\right)\leq 0
\end{equation*}
holds;
\item[(b)] if $\xi_n(\omega)\in
\partial\Theta$, the inequality
\begin{eqnarray*}
\min\left[
\mathrm{A}_{\widetilde{f}_{n},0}\left(\varphi_n\left(\omega,\tau_{n}(\omega),\xi_{n}(\omega)\right)\right)-D_t
\varphi_n\left(\omega,\tau_{n}(\omega),\xi_{n}(\omega)\right),\right.\\
\left.-\frac{\partial\varphi_n}{\partial
n}(\omega,\tau_{n}(\omega),\xi_{n}(\omega))-\widetilde{\phi}(\omega,\tau(\omega)_{n},\xi_{n}(\omega),\Psi_n(\omega,\tau_{n}(\omega),\xi_{n}(\omega)))\right]\leq
0
\end{eqnarray*}
holds.
\end{itemize}
On the other hand, since $v(\omega,\tau(\omega),\xi(\omega))>\widetilde{h}(\omega,\tau(\omega),\xi(\omega))$, it follows from \eqref{cv} that $v^n(\omega,\tau(\omega),\xi(\omega))>\widetilde{h}(\omega,\tau(\omega),\xi(\omega))$ for $n$ large enough such that passing to the limit in the two last inequalities, we get:
\begin{itemize}
 \item[(a)] if $\xi(\omega)\in\Theta$, the inequality
\begin{equation*}
\mathrm{A}_{\widetilde{f},0}\left(\varphi\left(\omega,\tau(\omega),\xi(\omega)\right)\right)-D_t
\varphi\left(\omega,\tau(\omega),\xi(\omega)\right)\leq 0
\end{equation*}
holds;
\item[(b)] if $\xi(\omega)\in
\partial\Theta$, the inequality
\begin{eqnarray*}
\min\left[
\mathrm{A}_{\widetilde{f},0}\left(\varphi\left(\omega,\tau(\omega),\xi(\omega)\right)\right)-D_t
\varphi\left(\omega,\tau(\omega),\xi(\omega)\right),\right.\\
\left.-\frac{\partial\varphi}{\partial
n}(\omega,\tau(\omega),\xi(\omega))-\widetilde{\phi}(\omega,\tau(\omega),\xi(\omega),\Psi(\omega,\tau(\omega),\xi(\omega)))\right]\leq
0
\end{eqnarray*}
holds.
\end{itemize}

\end{proof}

\noindent{\bf Acknowledgments}\newline This work is partially done when the first author was post doctoral
internship at Cadi Ayyad University of Marrakech. He would like
to express his deep gratitude to B. Boufoussi, Y. Ouknine and UCAM
Mathematics Department for their friendly hospitality. An anonymous
referee is also acknowledged for his comments, remarks and for a
significant improvement to the overall presentation of this paper.

\label{lastpage-01}
\end{document}